\documentclass[final,1p,times]{elsarticle}
\journal{}

\usepackage{amssymb}

\usepackage{amsthm}
\usepackage{amsmath,amssymb,amsopn,amsfonts,mathrsfs,amsbsy,amscd}
\usepackage{longtable}

\usepackage{hyperref}

\usepackage{xcolor}

\newcommand{\black}{\color{black}}

\newcommand{\In}{\{\cdot,\cdot\}}

\newcommand{\bb}{\bar{b}}

\newcommand{\prs}{\langle\cdot,\cdot\rangle}

\newcommand{\om}{\omega}
\newcommand{\esp}{\quad\mbox{and}\quad}

\def\br{[\cdot,\cdot]}
\newcommand{\A}{{\cal A}}

\newcommand{\Ha}{{\cal H}}

\newcommand{\h}{{\mathfrak{h}}}
\newcommand{\ad}{{\mathrm{ad}}}

\DeclareMathOperator{\asso}{Asso}

\DeclareMathOperator{\spa}{span}

\newcommand{\al}{\alpha}
\newcommand{\be}{\beta}

\newcommand{\la}{\lambda}

\newcommand{\de}{\delta}

\newcommand{\Ll}{{\mathrm{L}}}
\newcommand{\Rr}{{\mathrm{R}}}

\DeclareMathOperator{\En}{End}

\newtheorem{theo}{Theorem}[section]
\newtheorem{pr}[theo]{Proposition}
\newtheorem{Le}[theo]{Lemma}
\newtheorem{co}[theo]{Corollary}

\theoremstyle{definition}
\newtheorem{Def}[theo]{Definition}

\newtheorem{remark}[theo]{Remark}

\numberwithin{equation}{section}

\font\bb=msbm10

\def\K{\hbox{\bb K}}

\begin{document}

\begin{frontmatter}

\title{On anti-left-invariant  and left-invariant pseudo-Euclidean nearly associative algebras}

\author[CMUC] {Elisabete Barreiro\fnref{fund1}} 
\ead{mefb@mat.uc.pt}
	
\author[UL]{Sa\"id Benayadi} 
\ead{said.benayadi@univ-lorraine.fr}

\author[NYU]{Hamza El Ouali\fnref{fund2}}
\ead{hamza.el.ouali@nyu.edu}

 \author[CMUC] {Carla Rizzo\fnref{fund1}} 
\ead{carlarizzo@mat.uc.pt}   

\affiliation[CMUC]{
    organization={CMUC, Department of Mathematics, University of Coimbra},
    addressline={Largo D.
Dinis}, 
   city={Coimbra},
    postcode={3000-143}, 
    country={Portugal}}

\affiliation[UL]{
    organization={Universit\'e de Lorraine, Laboratoire IECL, CNRS-UMR 7502, UFR MIM},
    addressline={3 rue Augustin Fresnel, BP 45112}, 
   city={Metz},
    postcode={57073}, 
    country={France}}

\affiliation[NYU]{
    organization={Division of Science and Mathematics, New York University Abu Dhabi},
    addressline={Saadiyat Island}, 
   city={Abu Dhabi},
    postcode={P.O. Box 129188}, 
    country={United Arab Emirates}}

\fntext[fund1]{E. Barreiro and C. Rizzo acknowledge financial support by the Centre for Mathematics of the University of Coimbra (CMUC, https://doi.org/10.54499/UID/00324/2025) under the Portuguese Foundation for Science and Technology (FCT), Grants UID/00324/2025 and UID/PRR/00324/2025.}

\fntext[fund2]{H. El Ouali was supported by the grant NYUAD-065.}

\begin{abstract}

We introduce the notions of anti-left-invariant and left-invariant pseudo-Euclidean nearly associative algebras, which arise as nearly associative Levi-Civita products associated with pseudo-Euclidean Lie and Jordan algebras, respectively. We establish a correspondence between these classes of algebras and the Levi-Civita products of their associated Lie and Jordan structures.

For anti-left-invariant pseudo-Euclidean nearly associative algebras, we prove that they are nilpotent of index at most five and characterize them as Jacobi--Jordan-admissible nearly associative algebras. We further show that the associated pseudo-Euclidean Jacobi--Jordan algebras are cyclic. 
Motivated by the classical double extension of Medina and Revoy, we introduce a double extension procedure for this class of algebras and prove that every anti-left-invariant pseudo-Euclidean nearly associative algebra can be obtained from a trivial pseudo-Euclidean algebra by a finite sequence of such extensions.

For left-invariant pseudo-Euclidean nearly associative algebras, we prove that the associated pseudo-Euclidean Lie algebras are two-step solvable and cyclic. We then develop block, planar, and linear double extensions and show that every left-invariant pseudo-Euclidean nearly associative algebra can be recursively constructed from a quadratic commutative associative algebra by means of block double extensions. Moreover, we prove that over the field of real numbers, every such algebra can be recursively constructed from a quadratic commutative associative algebra using planar double extensions.

These recursive constructions provide a unified framework for describing and classifying pseudo-Euclidean nearly associative algebras in both the anti-left-invariant and left-invariant settings.
\end{abstract}

\end{frontmatter}

\section{Introduction}

The Levi-Civita connection is one of the fundamental objects in differential geometry. When a pseudo-Riemannian metric is left invariant on a Lie group, the Levi-Civita connection is completely determined by an algebraic product on the Lie algebra, known as the Levi-Civita product. This correspondence provides a powerful bridge between geometry and non-associative algebra, allowing geometric properties such as flatness and curvature to be translated into purely algebraic identities.


A \textit{pseudo-Riemannian Lie group} is a Lie group $G$ endowed with a left-invariant pseudo-Riemannian metric $\mu$.  The corresponding Lie algebra $(\mathfrak g=T_eG, \br)$ inherits the non-degenerate symmetric bilinear form $\langle\cdot,\cdot\rangle=\mu_e,$ and the triple $(\mathfrak g,\br,\prs)$ is called a \textit{pseudo-Euclidean Lie algebra}.

The Levi-Civita connection associated with $\mu$ defines a bilinear product $\star$ on $\mathfrak{g}$, called the \textit{Levi-Civita product}, given by Koszul’s formula:
\[
2 \langle u \star v, w \rangle
= \langle [u, v], w \rangle - \langle [u, w], v \rangle + \langle [w, v], u \rangle,
\qquad \text{for all }  u,v,w \in \mathfrak{g}.
\]
This product is uniquely determined by the metric and satisfies
\[
u \star v - v \star u = [u,v],
\qquad
\langle u \star v, w \rangle = -\langle v, u \star w \rangle,
\]
showing that the Lie bracket is recovered as the skew-symmetrization of the Levi-Civita product, while the left multiplication operators are skew-symmetric with respect to the metric.

The \textit{curvature} of a pseudo-Riemannian Lie group $(G, \mu)$ at the identity is given by  the curvature tensor $\mathfrak{R}$ defined by 
$$
\mathfrak{R}(u,v):=\Ll_{[u,v]}-[\Ll_u,\Ll_v], \qquad \text{for all }   u,v\in\mathfrak{g},
$$
where $\Ll_u(v)=u\star v$ denotes left multiplication by $u$.  Consequently, the curvature is entirely determined by the Levi-Civita product.

A pseudo-Riemannian Lie group $(G,\mu)$ is called \textit{flat} if its curvature vanishes identically; in this case, $(\mathfrak{g},\br, \prs)$ is called a \textit{flat pseudo-Euclidean Lie algebra}. 
The vanishing of the curvature is equivalent to the fact  that $\mathfrak{g}$, endowed with the Levi-Civita product $\star$, is a left-symmetric algebra, that is,
$$\asso(u, v, w) = \asso(v,u,w), \qquad \text{for all }   u,v,w\in\mathfrak{g},$$
where, as usual,  $\asso(u, v, w):=(u\star v)\star w-u\star(v\star w)$ denotes the associator. Thus, a flat pseudo-Euclidean Lie algebra can be viewed as a left-symmetric algebra endowed with a non-degenerate symmetric bilinear form for which the left multiplications are skew-symmetric.

The Levi-Civita product associated with pseudo-Euclidean Lie algebras has been extensively studied, particularly in the context of flat pseudo-Euclidean Lie algebras (see, for instance, \cite{BL3, Au, AM, BL1, BL2, Milnor}).

 More recently, this construction has been extended well beyond the Lie setting. It was first introduced for pseudo-Euclidean commutative algebras in \cite{Benayadi1}, then generalized to arbitrary non-associative algebras in \cite{ABE}, and subsequently to non-associative superalgebras in \cite{BBE}. These developments have led to further applications, notably in the study of pseudo-Euclidean Novikov algebras \cite{L,BEL} and left Leibniz (super)algebras \cite{BB,BBE}. 
Recall that a \textit{pseudo-Euclidean non-associative algebra} $(\mathcal{A}, \bullet, \prs)$ is a non-associative algebra $(\mathcal{A}, \bullet)$ endowed with a symmetric non-degenerate bilinear form.


Among the many varieties of non-associative algebras, \emph{nearly associative algebras} occupy a distinguished position. These are algebras $(\mathcal A,\bullet)$ satisfying
\[
(u\bullet v)\bullet w=v\bullet(w\bullet u),
\qquad \text{for all } u,v,w \in \mathcal{A}.
\]
As shown in \cite{BenBar}, they are simultaneously Lie-admissible and Jordan-admissible, meaning that the commutator and the symmetrized product naturally define a Lie and a Jordan structure, respectively. This dual nature makes nearly associative algebras an ideal setting for studying Levi-Civita products arising from Lie and Jordan structures simultaneously.

Nearly associative algebras form a particular class of the broader family of algebras satisfying multilinear identities involving three variables studied by Kosier and Osborn \cite{Osborn}, and were first introduced in \cite{Dasso}. Their structure was investigated in depth in \cite{BenBar}, where several algebraic properties were established. Further developments were obtained in \cite{Ivan} under the name of shift associative algebras, including a complete classification in dimensions up to four. More recently, the associated bialgebra structures and additional invariants were studied in \cite{BenBar1}.


Given that nearly associative algebras are both Lie-admissible and Jordan-admissible, it is therefore natural to study
 two classes of pseudo-Euclidean nearly associative algebras, corresponding respectively to the Lie and Jordan Levi-Civita products. 

Let $(\mathfrak{g}, \br, \prs)$ be a pseudo-Euclidean Lie algebra whose associated Levi-Civita product $\bullet$ defines a nearly associative algebra. We call the resulting structure
$(\mathfrak g,\bullet,\prs)$ an \emph{anti-left-invariant pseudo-Euclidean nearly associative algebra}, that is, a pseudo-Euclidean nearly associative algebra satisfying
\begin{equation*}
\langle u \bullet v, w \rangle = -\langle v, u \bullet w \rangle, \qquad \text{for all } u,v,w \in \mathfrak{g}.
\label{T1} 
\end{equation*}
Conversely, every anti-left-invariant pseudo-Euclidean nearly associative algebra arises in this way. Indeed, if $(\A,\bullet,\prs)$ is an anti-left-invariant pseudo-Euclidean nearly associative algebra, then its commutator algebra $\A^-$, that is, $\A$ endowed with the commutator product, is a pseudo-Euclidean Lie algebra, and the product $\bullet$ is precisely the Levi-Civita product associated with $(\A^-,\prs)$.

Similarly, let $(\mathfrak{J}, \{\cdot,\cdot\}, \prs)$ be a pseudo-Euclidean Jordan algebra whose Levi-Civita product $\bullet$ is nearly associative. We call the resulting structure $(\mathfrak{J}, \bullet, \prs)$  a \emph{left-invariant pseudo-Euclidean nearly associative algebra}, that is, a pseudo-Euclidean nearly associative algebra satis\-fying
\begin{equation*}
\langle u \bullet v, w \rangle = \langle v, u \bullet w \rangle, \qquad \text{for all } u,v,w \in \mathfrak{J}.
\label{T2}
\end{equation*}
 Conversely, every left-invariant pseudo-Euclidean nearly associative algebra is obtained from a pseudo-Euclidean Jordan algebra through its Levi-Civita product. Indeed, if $(\A,\bullet,\prs)$ is a left-invariant pseudo-Euclidean nearly associative algebra, then the commutative algebra $\A^+$, that is, $\A$ endowed with the symmetrized product, is a pseudo-Euclidean Jordan algebra, and the product $\bullet$ is precisely the Levi-Civita product associated with $(\A^+,\prs)$.

Therefore, the study of Levi-Civita products provides a unified framework that naturally connects the Lie and Jordan settings within nearly associative algebras. 
This perspective allows us to interpret anti-left-invariant and left-invariant pseudo-Euclidean nearly associative algebras as two complementary aspects of the same geometric-algebraic phenomenon.


 Motivated by this correspondence, this paper extends the theory of quadratic nearly associative algebras, namely pseudo-Euclidean nearly associative algebras where the form is invariant, developed in \cite{BenBar,BenBar1} to the anti-left-invariant and left-invariant pseudo-Euclidean case. 
Our main objective is to introduce the new classes of anti-left-invariant and left-invariant pseudo-Euclidean nearly associative algebras and establish the foundations of their theory. 

For anti-left-invariant pseudo-Euclidean nearly associative algebras, we prove that they are nilpotent of index at most five and are also Jacobi--Jordan-admissible nearly associative algebras. Jacobi-Jordan algebras, also known as \textit{Mock-Lie algebras}, have been extensively studied in the literature (see, for instance, \cite{Ba, Benayadi, Benayadi1, BE-El, burde,  OK}).
We further establish that the pseudo-Euclidean Jacobi-Jordan algebra associated with an anti-left-invariant pseudo-Euclidean nearly associative algebra is cyclic, a class of pseudo-Euclidean Jacobi-Jordan algebras introduced and studied in \cite{BE-El}.
We also show that every anti-left-invariant Euclidean or Lorentzian nearly associative algebra is trivial, whereas every algebra of signature $(2,n-2)$ is nilpotent of index at most three.
Motivated by the classical construction of Medina and Revoy \cite{Medina-Revoy}, we introduce a double extension procedure and prove that every anti-left-invariant pseudo-Euclidean nearly associative algebra can be obtained from a trivial pseudo-Euclidean algebra by a finite sequence of such extensions. As an application, we classify all anti-left-invariant pseudo-Euclidean nearly associative algebras of dimension at most six and signature $(2,n-2)$.

For left-invariant pseudo-Euclidean nearly associative algebras, we prove that the associated Lie algebras are solvable of index at most two. We also show that the pseudo-Euclidean Lie algebra associated with a left-invariant pseudo-Euclidean nearly associative algebra is cyclic, a class of pseudo-Euclidean Lie algebras introduced and investigated in \cite{ga} and further explored, for example, in \cite{{Abid}}.
Moreover, we establish that every left-invariant Euclidean nearly associative algebra is a quadratic commutative associative algebra.
We then introduce block, planar, and linear double extensions and show that every left-invariant pseudo-Euclidean nearly associative algebra can be constructed from a quadratic commutative associative algebra by a finite sequence of linear or planar double extensions. As a consequence, we classify all  left-invariant Lorentzian non-commutative nearly associative algebras of dimension at most four.


The paper is organized as follows. In Section~\ref{s2}, we recall the basic notions and preliminary results that will be used throughout the paper. Sections~\ref{s3}, \ref{s4} and~\ref{s:Class anti-left} are devoted to anti-left-invariant pseudo-Euclidean nearly associative algebras. In Section~\ref{s3}, we establish their fundamental structural properties and investigate their relationship with pseudo-Euclidean Jacobi--Jordan algebras, while Section~\ref{s4} develops the corresponding double extension theory and its applications to recursive constructions and classification results. Section~\ref{s:Class anti-left} is devoted to the complete classification of anti-left-invariant pseudo-Euclidean nearly associative algebras of dimension at most six with $\prs$ of signature $(2,n-2)$.
The second part of the paper, comprising Sections~\ref{s5}--\ref{s:class left-inv}, concerns left-invariant pseudo-Euclidean nearly associative algebras. Section~\ref{s5} investigates their structural properties and the associated pseudo-Euclidean Lie algebras, whereas Sections~\ref{s6} and~\ref{s7} introduce the corresponding double extension constructions and apply them to obtain recursive descriptions and classification results. Finally, Section~\ref{s:class left-inv} provides a complete classification of  left-invariant Lorentzian non-commutative nearly associative algebras of dimension at most four.
\black

\section{Preliminaries}\label{s2}

In this section, we recall the notions and results on pseudo-Euclidean non-associative algebras and nearly associative algebras that will be used throughout the paper, and we introduce two constructions that will play a central role in the sequel: annihilator extensions and almost semi-direct products.

Throughout this work, all vector spaces are assumed to be finite dimensional over a field $\mathbb{K}$ of characteristic zero.

A \emph{pseudo-Euclidean vector space} is a finite-dimensional $\mathbb{K}$-vector space $V$ endowed with a non-degenerate symmetric bilinear form
$\prs:V\times V\longrightarrow\mathbb{K}.$ When $\mathbb{K}=\mathbb{R}$, the form $\prs$ has a well-defined signature $(q,n-q)=(-\cdots-,+\cdots+)$, where $n=\dim V$. In particular, $(V,\prs)$ is called \emph{Euclidean} if its signature is $(0,n)$, and \emph{Lorentzian} if its signature is $(1,n-1)$. 
 Throughout the paper, whenever we refer to the signature of a pseudo-Euclidean vector space, we implicitly assume that the base field is $\mathbb{R}$.

Let $(V, \prs)$ be a pseudo-Euclidean vector space. A vector $u \in V$ is called  \emph{isotropic} if $\langle u, u \rangle = 0$. A family $(u_1, \dots, u_s)$ of vectors in $V$ is called  \emph{orthogonal} if $\langle u_i, u_j \rangle = 0$ for all $i, j = 1, \dots, s$ with $i \neq j$,  and \emph{orthonormal} if, in addition, $\langle u_i, u_i \rangle = 1$ for all $i = 1, \dots, s$.

A vector subspace $U$ of $V$ is called  \emph{non-degenerate} if $U \cap U^\perp = \{0\}$,  \emph{degenerate} if $U \cap U^\perp \neq \{0\}$, and  \emph{totally isotropic} if $\langle u, v \rangle = 0$ for all $u, v \in U$. Here $U^\perp$ denotes the orthogonal complement of $U$ with respect to $\prs$. Observe that we always have $\dim U + \dim U^\perp = \dim V$. Moreover, if $\mathbb{K} = \mathbb{R}$ and $U$ is totally isotropic, then $\dim U \leq \min(q, n-q)$. Consequently,$ \dim(U\cap U^\perp)\leq\min(q,n-q)$
for every vector subspace $U$ of a real pseudo-Euclidean vector space.

 Given another pseudo-Euclidean vector space $(W, \prs_W)$, a vector space isomorphism $\phi:V \to W$ is called  \emph{isometry} if $\langle \phi(u), \phi(v)\rangle_W=\langle u, v \rangle_V$ for all $u,v \in V$. 

For $x,y\in V$, let $x^*,y^*\in V^*$ denote the corresponding dual forms.
For $x^* \otimes y^* \in V^* \otimes V^*$ and $u \otimes v \in V \otimes V$, we define the natural pairing by
$\langle x^* \otimes y^*, \, u \otimes v \rangle := x^*(u) y^*(v),$
and the symmetric product $\odot$ on $V^*$ is given by
\[
x^* \odot y^* :=x^* \otimes y^* + y^* \otimes x^*.
\]

Given a linear endomorphism $D$ of $V$, there exists a unique linear map $D^*$, called the  \emph{adjoint of} $D$, that satisfies  $\langle  D(x),y \rangle = \langle  x,D^*(y)\rangle$    for all $x  , y \in V$.  If $D=D^*$, then $D$ is said to be  \emph{symmetric} with respect to $\prs$. If $D^*=-D$, then $D$ is said to be  \emph{skew-symmetric} with respect to $\prs$.

\smallskip
 
Let $(A, \cdot )$ be a non-associative algebra. As usual, we denote by $\Ll, \Rr : A \to \En(A)$ the left and right multiplication operators, respectively. More precisely,  for $u \in A$, these operators are defined by
\[
\Ll_u(v) := u \cdot v, \qquad 
\Rr_u(v) := v \cdot u, 
\] 
for any $v\in A$. We write the usual composition of endomorphisms by $\circ$.

The \textit{associator}  of $A$ is the trilinear map  $\mbox{Asso}: A \times A \times  A \longrightarrow  A$
defined by
\begin{equation*}
\begin{split}
\asso(u,v,w):=(u\cdot v)\cdot w-u\cdot(v\cdot w),
\end{split}
 \end{equation*}
for any $u,v,w\in A$. 

\begin{Def}\label{def-invariant}
Let $(A, \cdot)$ be a  non-associative algebra.  A bilinear form
$\prs :  A \times  A \longrightarrow  \mathbb{K}$ is said to be:
\begin{enumerate}
    \item  \emph{invariant} if
    \(
    \langle u \cdot v, w \rangle = \langle u, v \cdot w \rangle 
    \), for all $u, v, w \in A$;
      
    \item \emph{anti-left-invariant} if
    \(
    \langle u \cdot v, w \rangle = -\langle v, u \cdot w \rangle 
    \), for all $u, v, w \in A$;

      \item \emph{left-invariant} if
    \(
    \langle u \cdot v, w \rangle = \langle v, u \cdot w \rangle 
    \), for all $u, v, w \in A$.
\end{enumerate}
\end{Def}

\begin{Def}
Let $(A, \cdot)$ be a non-associative algebra endowed with a non-degenerate symmetric bilinear form 
$\prs$. The triple $(A,\cdot, \prs)$ is called a \emph{pseudo-Euclidean non-associative algebra} and $\prs$ is referred to as a \emph{pseudo-Euclidean structure} on $A$.
\end{Def}

\begin{Def}
     A pseudo-Euclidean non-associative algebra $(A, \cdot,\prs)$ is called:
     \begin{enumerate}
         \item \emph{quadratic} if the pseudo-Euclidean structure on $A$ is invariant, and $\prs$ is referred as \emph{invariant scalar product} on $(A, \cdot)$.
         \vspace{1mm}

         \item \emph{anti-left-invariant} if the pseudo-Euclidean structure on $A$ is anti-left-invariant, and $\prs$ is referred as \emph{anti-left-invariant scalar product} on $(A, \cdot)$.
         \vspace{1mm}

         \item \emph{left-invariant} if the pseudo-Euclidean structure on $A$ is left-invariant, $\prs$ is referred as \emph{left-invariant scalar product} on $(A, \cdot)$.
     \end{enumerate}
\end{Def}

\begin{Def}[\cite{BE-El}]
A pseudo-Euclidean non-associative algebra $(A, \cdot, \prs)$ is called \emph{cyclic} if the following condition holds 
\begin{equation}
\langle  u\cdot v , w \rangle + \langle  v \cdot w , u \rangle + \langle  w \cdot u , v \rangle = 0,
\label{cyclic}
\end{equation}
for any $u, v, w \in A$.
\end{Def}

\begin{Def}\label{def: flat}
Let $(A,\cdot)$ be a non-associative algebra. The \emph{curvature}  of $(A,\cdot)$ is the bilinear map
$
K:A\times A \longrightarrow \mathrm{End}(A) 
$
defined by
\[
K(u,v):=\Ll_{u\cdot v-v\cdot u}-[\Ll_u, \Ll_v],
\qquad \text{ for all  } u,v\in A.
\]
The algebra $( A,\cdot)$ is said \emph{flat} if its curvature vanishes identically.
\end{Def} 
It is well known that a non-associative algebra $(A,\cdot)$ is flat if and only if it is a left-symmetric algebra, that is, 
\begin{equation}\label{eq: flat}
    \asso(u, v, w) = \asso(v,u,w),
\end{equation}
for any $u, v, w \in A$.

The following result, proved in \cite{ABE}, provides the analogue of the Levi-Civita product for pseudo-Euclidean non-associative algebras.

\begin{pr}[\cite{ABE}]\label{LVCT} Let $(A,\cdot,\prs)$ be a pseudo-Euclidean non-associative algebra. Then there exists a unique couple of products $(\star, \diamond)$ on $A$ satisfying, for any $u,v,w\in A$, 
	\begin{align}
	     \label{torsion}
		&u\star v+v\diamond u=u \cdot v \\ \label{compatible}&\langle u\diamond v,w\rangle= \langle v,u\star w\rangle.
	\end{align}
Moreover, these products are characterized by Koszul’s formula:
	\begin{equation}\label{lv}
		2\langle u\star v,w\rangle=\langle u \cdot v,w\rangle-\langle v \cdot w,u\rangle+\langle w \cdot u,v\rangle.
	\end{equation}
In addition, if the product $\cdot$ is anti-commutative (resp. commutative) then $\star=-\diamond$ (resp. $\star=\diamond$). 
	The product $(\star,\diamond)$ is called  the \emph{Levi-Civita product} associated to $(A, \cdot,\prs)$.
\end{pr}

\smallskip

The main objects of this paper are left-invariant and
anti-left-invariant pseudo-Euclidean nearly associative
algebras. We therefore recall below the basic definitions and properties of
nearly associative algebras that will be used throughout the
paper.

\begin{Def}
A non-associative algebra $(\A, \bullet)$ is said to be \emph{nearly associative} if, for all $u,v,w \in \A$, the following identity holds:
\begin{equation}\label{nearly-associative}
(u\bullet v)\bullet w = v\bullet (w\bullet u).
\end{equation}
\end{Def}

\begin{remark}
If $(\A, \bullet)$ is a nearly associative algebra, then for all $u,v \in \A$, the following relations between the left and right multiplication operators hold:
\begin{align}
\Ll_u \circ \Ll_v &= \Rr_v \circ \Rr_u, \label{left-right-1}\\
\Ll_u \circ \Rr_v &= \Ll_{v\bullet u}, \label{left-right-2}\\
\Rr_u \circ \Ll_v &= \Rr_{u\bullet v}. \label{left-right-3}
\end{align}
Conversely, if a non-associative algebra $(\A, \bullet)$ satisfies any one of the Conditions \eqref{left-right-1}, \eqref{left-right-2} or \eqref{left-right-3}, then $(\A, \bullet)$ is nearly associative.
\end{remark}

Before proceeding further, let us recall some fundamental definitions on non-associative algebras needed in the sequel.

Let $(A,\cdot)$ be a non-associative algebra. On the underlying vector space $A$, we may introduce the following two bilinear operations:
\begin{equation}\label{skew-sym}
[u,v] :=  u\cdot v - v\cdot u , 
\qquad 
\{u,v\} :=  u\cdot v + v\cdot u ,
\end{equation}
for all $u,v \in A$.  
We denote by $A^- := (A, [\,,\,])$ the algebra endowed with the anti-commutative product,  
and by $A^+ := (A, \{\,,\,\})$ the one endowed with the commutative product.  
The product of $A$ can be expressed as
\[
2u\cdot v = [u,v] + \{u,v\}, \qquad \text{for all } u,v \in A.
\]

\begin{remark}
In the papers \cite{BenBar, BenBar1} dealing with  nearly associative algebras,  the two bilinear operations presented in \eqref{skew-sym} are multiplied by the factor of $\frac{1}{2}$. 
But in the context of Riemannian geometry, it is often more natural to work simply with the commutator and the anti-commutator rather than  to use the factor  $\frac{1}{2}$. 
Furthermore, the results presented in the aforementioned papers on nearly associative algebras, which we shall use here, continue to be valid in this context as well.
So, in the present paper we will not consider the factor $ \tfrac{1}{2} $ so that the article remains more naturally oriented toward algebra and Riemannian geometry.
\end{remark}

\smallskip

Recall that an algebra $(\mathfrak{g}, \br)$ is a \emph{Lie algebra} if, for all $u,v,w \in \mathfrak{g}$, the following two conditions are satisfied:
\begin{align*}
[u,v] = -[v,u] &, \\
[u,[v,w]] + [v,[w,u]] + [w,[u,v]] = 0 & \quad \text{(Jacobi identity)}.
\end{align*}
The \emph{center} of a Lie algebra $(\mathfrak{g}, \br)$ is defined by
\[
Z(\mathfrak{g}) := \{ u \in \mathfrak{g} \mid [u,v] = 0,\ \forall\, v \in \mathfrak{g} \}.
\]

An algebra $(J, \{\cdot,\cdot\})$ is   a \emph{Jordan algebra} if, for all $u,v \in J$, the following two conditions hold:
\begin{align*}
\{u,v\} = \{v,u\} &, \\
\{\{u,v\},u^2\} = \{u,\{v,u^2\}\}& \quad \text{(Jordan identity)},
\end{align*}
where $u^2 := \{u,u\}$.

An algebra $(J, \{\cdot,\cdot\})$ is   a \emph{Jacobi-Jordan algebra} if, for all $u,v,w \in J$, 
\begin{align*}
\{u,v\} &= \{v,u\}, \\
\{\{u,v\},w\} + \{\{v,w\},u\} + \{\{w,u\},v\} &= 0 \quad \text{(Jacobi identity)}.
\end{align*}
The class of Jacobi-Jordan algebras forms a subclass of  Jordan algebras. Moreover, Jacobi-Jordan algebras are known to be nilpotent (see \cite{Ba, burde, OK}).

\begin{Def}
A non-associative algebra $(A, \cdot)$ is said to be:
\begin{enumerate}
    \item  \emph{Lie-admissible} if the algebra $A^{-}$ is a Lie algebra;
    \item   \emph{Jordan-admissible} if the algebra $A^{+}$ is a Jordan algebra;
    \item   \emph{Jacobi–Jordan-admissible} if the algebra $A^{+}$ is a Jacobi-Jordan algebra.
\end{enumerate}
\end{Def}
In the Lie-admissible case, the Lie algebra $A^{-}$ is called the \emph{underlying Lie algebra} of $(A, \cdot)$, whereas in the  Jordan-admissible case (resp. Jacobi–Jordan-admissible case), the Jordan algebra (resp. the Jacobi–Jordan algebra) $A^{+}$ is called the \emph{underlying Jordan algebra} (resp. underlying Jacobi–Jordan algebra) of $(A, \cdot)$.

\smallskip

We say that the algebra $(A, \cdot)$ is \emph{non-trivial} if its multiplication is non-trivial, that is, $A \cdot A \neq \{0\}$, where 
$$
A \cdot A= \spa \{u \cdot v \mid u,v \in A\}.
$$
The subspace $A\cdot A$ is a two-sided ideal of $A$, usually called the \emph{derived algebra} of $A$. Clearly, $(A,\cdot)$ is called \emph{trivial} if its multiplication is identically zero, that is, $A\cdot A=\{0\}$.

For $i\geq 1$, let $A^i$ denote the subspace of $A$ spanned by all products of $i$ elements of $A$ (with arbitrary parenthesization). Clearly, $A^2=A\cdot A$.

The algebra $A$  is called \emph{nilpotent} if  $A^n = \{0\} $ for some $n\geq 1$. More specifically, $A$ is said to be \emph{nilpotent of index $n$} if $A^n = \{0\} $ and $A^{n-1} \neq \{0\}$.
In general, $A$ is said to be \emph{nilpotent of index at most $n$} if $A^n = \{0\}$, but possibly it could vanish even earlier, that is, $A^k = \{0\}$ for some $k < n$.
 
 Next, define $A^{(1)}:=A^2$ and $A^{(i+1)}:=(A^{(i)})^2$ for $i\geq 1$. The algebra $A$ is called \emph{solvable} if $A^{(n)}=\{0\}$ for some $n\geq 1$. 
More precisely, $A$ is said to be \emph{solvable of index $n$} if $A^{(n)} = \{0\} $ and $A^{(n-1)} \neq \{0\}$.
More generally, $A$ is said to be \emph{solvable of index at most $n$} if $A^{(n)} = \{0\}$, but possibly it could vanish even earlier, that is, $A^{(k)} = \{0\}$ for some $k < n$. 
Clearly, any nilpotent algebra is solvable.

\begin{pr}[\cite{BenBar}, Propositions~4.2 and~4.5]\label{Lie-Jordan}
Let $(\A, \bullet)$ be a nearly associative algebra. Then the following statements hold:
\begin{enumerate}[$(i)$]
    \item  $\A^{-}$ is a solvable Lie algebra;
    \item $\A^{+}$ is a Jordan algebra.
\end{enumerate}
\end{pr}

\smallskip

Now, we present  the notion of annihilator extension of a nearly associative algebra.
Let $(\A, \bullet_{\mathcal{A}})$ be a nearly associative algebra, let $V$ be a vector space, and let 
$\varphi : \A \times \A \longrightarrow V$
be a bilinear map. On the vector space 
$
\widetilde{\A} := \A \oplus V,
$
we define a new product  by
\[
(u + x) \widetilde{\bullet} (v + y) := u \bullet_{\mathcal{A}} v + \varphi(u,v), \qquad \text{ for all }\, u,v \in \A,\; x,y \in V.
\]
It is easy to check that  $(\widetilde{\A}, \widetilde{\bullet})$ is a nearly associative algebra if and only if the bilinear map $\varphi$ satisfies the condition
\begin{equation}\label{ex-central}
\varphi(u \bullet_{\mathcal{A}} v, w) = \varphi(v, w \bullet_{\mathcal{A}} u), \qquad \text{ for all }\, u,v,w \in \A.
\end{equation}

\begin{Def}
If a bilinear map $\varphi$ satisfies the Condition \eqref{ex-central}, then the nearly associative algebra $(\widetilde{\A}, \widetilde{\bullet})$ is called the \emph{annihilator extension} of the nearly associative algebra $(\A, \bullet_{\mathcal{A}})$ by means of $\varphi$.
\end{Def}

\medskip

\noindent\textbf{The one-dimensional case.}  
If $\dim V = 1$, we put $V := \K e$, then the bilinear map $\varphi$ can be written as
\[
\varphi(u,v) = \phi(u,v)e,
\]
where $\phi : \A \times \A \to \K$ is a bilinear form.  
In this case, the product on $\widetilde{\A} = \A \oplus \K e$ becomes
\[
(u + \alpha e) \widetilde{\bullet} (v + \beta e) = u \bullet_{\mathcal{A}} v + \phi(u,v)e, 
\qquad \text{ for all }\, u,v \in \A,\; \alpha,\beta \in \K.
\]
The Condition \eqref{ex-central} then reduces to
\begin{equation}\label{line-central}
\phi(u \bullet_{\mathcal{A}} v, w) = \phi(v, w \bullet_{\mathcal{A}} u), \qquad \text{ for all }\, u,v,w \in \A.
\end{equation}

\begin{Def}
If a bilinear form $\phi$ satisfies the Condition \eqref{line-central}, then the nearly associative algebra $(\widetilde{\A}, \widetilde{\bullet})$ is called the \emph{one-dimensional annihilator extension} (or \emph{annihilator extension by line}) of the nearly associative algebra $(\A, \bullet_{\mathcal{A}})$ by means of $\phi$.
\end{Def}

Lastly, we introduce the notion of  almost semi-direct product  of nearly associative algebras.    Let $(\A, \cdot)$ and $(\Ha, \ast)$ be two nearly associative algebras.  
Let 
$
F, G : \Ha \longrightarrow \En(\A)$ be two linear maps and 
$L : \Ha  \times \Ha \longrightarrow \A
$
be a bilinear map. On the vector space
$
\overline{\A} := \Ha \oplus \A,
$
we define a product $\bullet$ by
\[
u \bullet v = u \cdot v, \quad
u \bullet X = F(X)(u),\quad
X \bullet u = G(X)(u), \quad
X \bullet Y = X \ast Y + L(X,Y),
\]
for all $u,v \in \A, \; X,Y \in \Ha.$
Then, by straightforward computations, $(\overline{\A}, \bullet)$ is a nearly associative algebra if and only if the following compatibility conditions hold:
\begin{equation}
\label{conditions-semi}
\begin{cases}
F(X)(u \cdot v) = v \cdot G(X)(u),\quad 
F(X)(u) \cdot v = G(X)(v \cdot u), \\[0.3em]
u \cdot F(X)(v) = G(X)(u) \cdot v, \quad
F(Y) \circ F(X) = G(X) \circ G(Y), \\[0.3em]
F(X) \circ G(Y) = \Rr_{L(X,Y)} + F(X \ast Y), \quad
G(Y) \circ F(X) = \Ll_{L(X,Y)} + G(X \ast Y), \\[0.3em]
L(X \ast Y, Z) + F(Z)\big(L(X,Y)\big) = L(Y, Z \ast X) + G(Y)\big(L(Z,X)\big),
\end{cases}
\end{equation}
for all $u,v \in \A$ and $X,Y,Z \in \Ha$.

\begin{Def}
If the Conditions \eqref{conditions-semi} are satisfied, the nearly associative algebra $(\overline{\A}, \bullet)$ is called an \textit{almost semi-direct product} of the nearly associative algebra $(\A, \cdot)$ by the nearly associative algebra $(\Ha, \ast)$, by means of  $(F, G, L)$. We call the tuple $(\Ha, \A, F, G, L)$ a \textit{context of an almost semi-direct product}  of nearly associative algebras.
\end{Def}

\medskip

\noindent\textbf{The one-dimensional case.}  
If $\dim \Ha = 1$, say $\Ha = \K d$, then we can write
\[
F(d)=\delta, \qquad G(d)=D, \qquad d\ast d=\lambda d, \qquad B(d,d)=a_0,
\]
where $\delta, D \in \mathrm{End}(\A)$, $a_0 \in \A$ and $\lambda \in \K$.  
The product on $\overline{\A} = \A \oplus \K d$ is then given by
\[
u \bullet v = u \cdot v, \qquad
u \bullet d = \delta(u), \qquad
d \bullet u = D(u), \qquad
d \bullet d = \lambda d + a_0,
\]
for all $u,v \in \A$, with $a_0 \in \A$. In this case, $(\overline{\A}, \bullet)$ is a nearly associative algebra if and only if the following conditions hold:
\begin{equation} 
\label{line-semi}
\begin{cases}
\delta(u \cdot v) = v \cdot D(u), \quad
\delta(u) \cdot v = D(v \cdot u), \\[0.3em]
u \cdot \delta(v) = D(u) \cdot v, \quad
\delta^2 = D^2, \\[0.3em]
\delta \circ D = R_{a_0} + \lambda\,\delta, \quad
D \circ \delta = L_{a_0} + \lambda\,D, \\[0.3em]
(\delta-D)(a_0) =0,
\end{cases}
\end{equation}
for all $ u,v \in \A$.
\begin{Def}
If the Conditions \eqref{line-semi}   hold, the nearly associative algebra $(\overline{\A}, \bullet)$ is called an \textit{almost semi-direct product} of the nearly associative algebra $(\A, \cdot)$ by means of $(\delta, D, a_0, \lambda)$.  
In this case,  we call the tuple $(\Ha, \A, \de, D, a_0, \la)$ a \textit{context of an almost semi-direct product}  of nearly associative algebras.
\end{Def}

\section{Anti-left-invariant pseudo-Euclidean nearly associative algebras}\label{s3}

In this section, we study anti-left-invariant pseudo-Euclidean nearly associative algebras. 
We show that every such algebra is nilpotent of index at most $5$.
Moreover, we prove that the commutative algebra associated to it is a Jacobi-Jordan algebra.
Actually, the problem under consideration admits two equivalent formulations, which we describe below.

\medskip
\noindent \textbf{Approach 1.}
Let $(\mathfrak g,[\, ,\, ],\prs)$ be a pseudo-Euclidean Lie algebra and let $\bullet$ be its Levi-Civita product. We investigate those pseudo-Euclidean Lie algebras for which $(\mathfrak g,\bullet)$ is a nearly associative algebra. Observe that the Levi-Civita product is automatically skew-symmetric with respect to $\prs$, that is,
\[
\langle u\bullet v,w\rangle = -\langle v, u\bullet w\rangle,
\qquad \text{ for all }\, u,v,w\in\mathfrak g.
\]
Thus, the condition that $(\mathfrak{g}, \bullet)$ is a nearly associative algebra is equivalent to requiring that $(\mathfrak g,\bullet, \prs)$ be an anti-left-invariant pseudo-Euclidean nearly associative algebra.

\medskip
\noindent \textbf{Approach 2  (Conversely).}
Let $(\A,\bullet, \prs)$ be an anti-left-invariant pseudo-Euclidean nearly associative algebra, which in particular means that $\prs$ satisfies
\begin{equation} 
\label{lcs}
\langle u\bullet v,w\rangle + \langle v, u\bullet w\rangle = 0,
\qquad \text{ for all }\, u,v,w\in\A. 
\end{equation}
By Proposition \ref{Lie-Jordan}, $\A^-=(\A, \br)$ is a Lie algebra. Hence, $(\A^-, \prs )$ is a pseudo-Euclidean Lie algebra satisfying Condition \eqref{lcs}. Therefore, by the uniqueness of the Levi-Civita product, it follows that the product $\bullet$ coincides with the Levi-Civita product associated with $(\A^{-},\prs)$.

\medskip
\noindent \textbf{Conclusion.}
The study of pseudo-Euclidean Lie algebras whose Levi-Civita product is nearly associative is equivalent to the study of anti-left-invariant pseudo-Euclidean nearly associative algebras.

\begin{Le}\label{Le1}
Let $(\A, \bullet, \prs)$ be an anti-left-invariant pseudo-Euclidean nearly associative algebra. Then, for all $u, v, w \in \A$, we have
\begin{equation} \label{eq:anti-left-lemma}
v \bullet (w \bullet u) = (u \bullet v) \bullet w =  - (v \bullet u) \bullet w = - u \bullet (w \bullet v).
\end{equation}
\end{Le}

\begin{proof}
Let $u, v, w, z \in \A$. Applying the nearly associativity identity, it is clear that 
\(
(u \bullet v) \bullet w = v \bullet (w \bullet u) \) and \(
(v \bullet u) \bullet w = u \bullet (w \bullet v).
\)
Using the anti-left-invariance of $\prs$ and the nearly associativity identity, we have 
\begin{align*}
\langle (u \bullet v) \bullet w, z \rangle 
= -\langle w, (u \bullet v) \bullet z \rangle  
&\overset{\eqref{nearly-associative}}{=} -\langle w, v \bullet (z \bullet u) \rangle  = -\langle z \bullet (v \bullet w), u \rangle \\
&\overset{\eqref{nearly-associative}}{=} -\langle (w \bullet z) \bullet v, u \rangle  = \langle v, (w \bullet z) \bullet u \rangle  \\
&\overset{\eqref{nearly-associative}}{=} \langle v, z \bullet (u \bullet w) \rangle = \langle u \bullet (z \bullet v), w \rangle  \\
&\overset{\eqref{nearly-associative}}{=} \langle (v \bullet u) \bullet z, w \rangle  = -\langle z, (v \bullet u) \bullet w \rangle.
\end{align*}
Since $\prs$ is non-degenerate, it follows that
\(
(u \bullet v) \bullet w = - (v \bullet u) \bullet w.
\)
Thus, we obtain the desired identities~\eqref{eq:anti-left-lemma}.
\end{proof}

\begin{pr}
    Let $(\A, \bullet, \prs)$ be an anti-left-invariant pseudo-Euclidean nearly associative algebra. Then,   $(\A, \bullet )$  is flat if and only if $\A^-$ is nilpotent of index at most $3$.
\end{pr}
\begin{proof}
For $u,v,w \in \A$, we have
\begin{align*}
\mathrm{Asso}(u,v,w) - \mathrm{Asso}(v,u,w)
&= (u\bullet v)\bullet w - u\bullet (v\bullet w)   - (v\bullet u)\bullet w + v\bullet (u\bullet w) \\
&\overset{\eqref{skew-sym}}{=} [u,v]\bullet w - u\bullet (v\bullet w) + v\bullet (u\bullet w)\\
&\overset{\eqref{eq:anti-left-lemma}}{=} [u,v]\bullet w + w\bullet (v\bullet u) - w\bullet (u\bullet v) \\
&= [u,v]\bullet w - w\bullet [u,v] \\
&= [[u,v],w].
\end{align*}
That is, $\mathrm{Asso}(u,v,w) - \mathrm{Asso}(v,u,w) = [[u,v],w]$.
Consequently, by Condition \ref{eq: flat},  $(\A, \bullet)$  is flat if and only if $\A^{-}$ is nilpotent of index at most $3$.
\end{proof}

\begin{theo}\label{nil5}
Let $(\A, \bullet, \prs)$ be an anti-left-invariant pseudo-Euclidean nearly associative algebra. Then $(\A, \bullet)$ is nilpotent of index at most  $5$.
\end{theo}

\begin{proof}
Let $u, v, w, z \in \A$. We have
\begin{align*}
\displaystyle \Ll_{(u \bullet v) \bullet (w \bullet z)} 
&\overset{\eqref{left-right-2}}{=} \Ll_{w \bullet z} \circ \Rr_{u \bullet v}  
\overset{\eqref{left-right-3}}{=} \Ll_{w \bullet z} \circ \Rr_u \circ \Ll_v \\
&\overset{\eqref{left-right-2}}{=} \Ll_{u \bullet (w \bullet z)} \circ \Ll_v  \overset{\eqref{nearly-associative}}{=} \Ll_{(z \bullet u) \bullet w} \circ \Ll_v \\
&\overset{\eqref{eq:anti-left-lemma}}{=} -\Ll_{(u \bullet z) \bullet w} \circ \Ll_v  
\overset{\eqref{left-right-2}}{=} -\Ll_w \circ \Rr_{u \bullet z} \circ \Ll_v \\
&\overset{\eqref{left-right-3}}{=} -\Ll_w \circ \Rr_u \circ \Ll_z \circ \Ll_v 
\overset{\eqref{left-right-1}}{=} -\Ll_w \circ \Rr_u \circ \Rr_v \circ \Rr_z \\
&\overset{\eqref{left-right-2}}{=} -\Ll_{u \bullet w} \circ \Rr_v \circ \Rr_z 
\overset{\eqref{left-right-2}}{=} -\Ll_{v \bullet (u \bullet w)} \circ \Rr_z \\
&\overset{\eqref{eq:anti-left-lemma}}{=} \Ll_{w \bullet (u \bullet v)} \circ \Rr_z 
\overset{\eqref{left-right-2}}{=} \Ll_{z \bullet (w \bullet (u \bullet v))}.
\end{align*}
On the other hand, using the nearly-associative property and the relation \eqref{eq:anti-left-lemma}, we also have 
\[
z \bullet (w \bullet (u \bullet v)) \overset{\eqref{nearly-associative}}{=} ((u \bullet v) \bullet z) \bullet w \overset{\eqref{eq:anti-left-lemma}}{=} - (z \bullet (u \bullet v)) \bullet w \overset{\eqref{nearly-associative}}{=} - (u \bullet v) \bullet (w \bullet z).
\]
Thus,
$ 
\Ll_{(u \bullet v) \bullet (w \bullet z)}  = -\Ll_{(u \bullet v) \bullet (w \bullet z)},
$
which implies
\[
\Ll_{(u \bullet v) \bullet (w \bullet z)} = 0.
\]
Then, for all $t \in \A$, we obtain
\[
((u \bullet v) \bullet (w \bullet z)) \bullet t
=0.
\]
Since $(\A,\bullet)$ is a nearly associative algebra, by applying \eqref{nearly-associative} it follows that
$\A^{5}=\{0\}$, and therefore $(\A,\bullet)$ is a nilpotent algebra of index at most $5$.
\end{proof}

From Proposition \ref{Lie-Jordan}, it is known  that every nearly associative algebra is Jordan-admissible. Within our framework, we strengthen this result by proving that any anti-left-invariant pseudo-Euclidean nearly associative algebra is Jacobi–Jordan admissible, as stated in the following proposition.

\begin{pr}\label{Jacobi-Jordan}
Let $(\A, \bullet, \prs)$ be an anti-left-invariant pseudo-Euclidean algebra.  
Then  $\A^+$ is a Jacobi-Jordan algebra.  
Moreover, $\A^+$ is nilpotent of index at most $5$.
\end{pr}
\begin{proof} To show that  $\A^+$ is a Jacobi-Jordan algebra 
we just have to verify the Jacobi identity for the commutative product $\{\, , \, \}$.  Let $u, v, w \in \A$, according to Lemma~\ref{Le1}, we have
\[
\{u, v\} \bullet w =  (u \bullet v + v \bullet u) \bullet w = 0,
\]
and hence
\begin{align*}
\{\{u,v\},w &\}  +  \{\{v,w\},u\} + \{\{w,u\},v\}
\\&  =   w\bullet \{u,v\} + u\bullet \{v,w\} + v\bullet \{w,u\}    \\
& =   w\bullet (u\bullet v) + w\bullet(v\bullet u) + u\bullet(v\bullet w) + u\bullet(w\bullet v) + v\bullet(w\bullet u) + v\bullet(u\bullet w) \\
& \overset{\eqref{eq:anti-left-lemma}}{=}    -v\bullet(u\bullet w) - u\bullet(v\bullet w) + u\bullet(v\bullet w) - v\bullet(w\bullet u) + v\bullet(w\bullet u) + v\bullet(u\bullet w)  \\
&  = 0.
\end{align*}
Therefore, $\A^+$ satisfies the Jacobi identity and, as a consequence, it is  a Jacobi-Jordan algebra.

Finally, according  to Theorem \ref{nil5} the anti-left-invariant pseudo-Euclidean nearly associative algebra $(\A, \bullet, \prs)$ is nilpotent of index at most $5$,   consequently, $\A^+$ is nilpotent of index at most $5$.    
\end{proof}

In \cite{BE-El}, the authors studied cyclic pseudo-Euclidean Jacobi-Jordan algebras. In this context, we show that any anti-left-invariant pseudo-Euclidean nearly associative algebra naturally gives rise to a cyclic pseudo-Euclidean Jacobi–Jordan algebra structure on $(\A^+, \prs)$.
\begin{pr}
Let $(\A, \bullet, \prs)$ be an anti-left-invariant pseudo-Euclidean nearly associative algebra. Then the  pseudo-Euclidean Jacobi–Jordan algebra  $(\A^+, \prs)$ is cyclic. Moreover, the Levi-Civita product $\star$ associated with $(\A^+, \prs)$ is given by
\[
\langle u \star v, w \rangle = -\langle u, \{v, w\} \rangle, \quad \text{for all } u,v,w \in \A.
\]
\end{pr}

\begin{proof}
Let $u,v,w \in \A$. By definition of the  product $\{ , \} $ we have
\[
\langle \{u,v\}, w \rangle =   \langle u \bullet v, w \rangle + \langle v \bullet u, w \rangle .
\]
Since $\prs$ is anti-left-invariant, it follows  
\begin{align*}
\langle  \{u,v\} , w \rangle +& \langle  \{v,w\} , u \rangle + \langle  \{w,u\} , v \rangle
\\  &=  \langle u \bullet v, w \rangle + \langle v \bullet u, w \rangle 
+\langle v \bullet w, u \rangle + \langle w \bullet v, u \rangle
+\langle w \bullet u, v \rangle + \langle u \bullet w, v \rangle   \\
& = \langle u \bullet v, w \rangle 
+ \langle v \bullet u, w \rangle 
-\langle v \bullet u, w \rangle 
+ \langle w \bullet v, u \rangle
- \langle w \bullet v, u \rangle 
- \langle u \bullet v, w \rangle   \\
&  = 0,
\end{align*}
that is,  $\prs$ is cyclic with respect to the product $\In$, which  shows that the pseudo-Euclidean Jacobi-Jordan algebra $(\A^+, \prs)$ is cyclic.

Now, by Proposition~\ref{LVCT}, the Levi-Civita product $\star$ associated to $(\A^+, \prs)$ is given by Koszul’s formula:  
\[
2 \langle u \star v, w \rangle = \langle \{u, v\}, w \rangle - \langle \{v, w\}, u \rangle + \langle \{w, u\}, v \rangle.
\]
Using the cyclicity of $\prs$ with respect to $\In$, we get 
$2 \langle u \star v, w \rangle  =   -2 \langle \{v, w\}, u \rangle,$
which proves the stated formula
\[
\langle u \star v, w \rangle = -\langle u, \{v, w\} \rangle,
\]
completing the proof.
\end{proof}

\smallskip

Let $(\A, \bullet)$ be a nearly associative algebra.
 We define the left and right annihilators and the two-sided annihilator of $(\A, \bullet)$ as:
$$
N_\ell(\A) := \{ u \in \A \mid \Ll_u = 0 \}, \quad
N_r(\A) := \{ u \in \A \mid \Rr_u = 0 \}, \quad
N(\A) := N_\ell(\A) \cap N_r(\A).
$$
If $(\A, \bullet, \prs)$ is a left-invariant pseudo-Euclidean nearly associative algebra, then it is easy to verify that 
\begin{equation} \label{R1}
    N_r(\A) = (\A \bullet \A)^\perp.
\end{equation}

\begin{pr}\label{pr9}
Let $(\A, \bullet, \prs)$ be a non-trivial anti-left-invariant pseudo-Euclidean nearly associative algebra. Then the subspace $\A \bullet \A$ and $N(\A)$ are degenerate.
In particular, any anti-left-invariant Euclidean nearly associative algebra must be trivial.
\end{pr}
\begin{proof}
According to Theorem~\ref{nil5}, $(\A, \bullet)$ is nilpotent. Hence, since $\A$ is a non-trivial algebra, i.e., $\A \bullet \A \neq \{0\}$, it follow that $(\A \bullet \A) \cap N(\A)\neq \{0\}$.
Using Condition \eqref{R1}, we  have $N(\A) \subseteq N_r(\A)=(\A\bullet \A)^\bot$, so it follows that
\[
(\A \bullet \A) \cap N(\A) \subseteq (\A \bullet \A) \cap (\A \bullet \A)^\perp \esp (\A \bullet \A) \cap N(\A) \subseteq N(\A)^\bot \cap N(\A) .
\]
Thus, $\A \bullet \A$ and $N(\A)$  are degenerate.

In the Euclidean case, the bilinear form is positive-definite or negative-definite, so no non-zero subspace can be degenerate. Therefore,  the algebra must be trivial.
\end{proof}

\section{Double extension of anti-left-invariant pseudo-Euclidean nearly associative algebras}\label{s4}

In this section, we introduce the notion of  double extension  of anti-left-invariant pseudo-Euclidean nearly associative algebras.    
We show that any anti-left-invariant pseudo-Euclidean nearly associative algebra can be obtained through successive double extensions, starting from a trivial pseudo-Euclidean  algebra.  
Furthermore, we prove that any anti-left-invariant Lorentzian algebra is necessarily a trivial algebra.

Let $(\mathcal{A}, \bullet_{\mathcal{A}}, \prs_{\mathcal{A}})$ be an anti-left-invariant pseudo-Euclidean nearly associative algebra. Consider the one-dimensional vector space $\h := \mathbb{K} d$ and let $\h^* := \mathbb{K} e$ denote its dual space.  
In order to introduce the process of double extension by $\h$, we consider two linear maps $\delta, D : \mathcal{A} \rightarrow \mathcal{A}$ such that $D$ is skew-symmetric with respect to $\prs_{\mathcal{A}}$, 
and an element $a_0 \in \mathcal{A}$. We extend the linear maps $\delta$ and $D$ to linear maps $\widetilde{\delta}, \widetilde{D} : \mathcal{A} \oplus \h^* \rightarrow \mathcal{A} \oplus \h^*$ as follows:   
\begin{equation}\label{eq:double-tilde}
\widetilde{\delta}(u + \alpha e) := \delta(u), \qquad \widetilde{D}(u + \alpha e) := D(u) - \langle a_0, u \rangle_{\mathcal{A}} \, e,
\end{equation}
for all $u \in \mathcal{A}$ and $\alpha \in \mathbb{K}$.

\medskip

\begin{Le} \label{Lemma1}
The vector space $\widetilde{\A} := \mathcal{A} \oplus \h^*$ equipped with the bilinear product
\[
(u + \alpha e) \widetilde{\bullet} (v + \beta e) := u \bullet_{\mathcal{A}} v - \langle \delta(u), v \rangle_{\mathcal{A}} e, \qquad \text{for all } u, v \in \mathcal{A} \text{ and } \alpha, \beta \in \mathbb{K},
\]
is a nearly associative algebra if and only if
\begin{equation}\label{Le-centrale}
\delta \circ \Ll_u = \Rr_u^* \circ \delta, \qquad \text{for all } u \in \mathcal{A}, 
\end{equation}
where  the  left and right multiplication operators are considered on $  \mathcal{A}  $.
\end{Le} 

\begin{proof}
Set $\phi(u,v) := -\langle \delta(u), v \rangle_{\mathcal{A}}$ for any $u, v \in \mathcal{A}$. Then for all $u, v, w \in \mathcal{A}$, we have:
\[
\phi(u \bullet_{\mathcal{A}} v, w) - \phi(v, w \bullet_{\mathcal{A}} u) = -\langle \delta(u \bullet_{\mathcal{A}} v), w \rangle_{\mathcal{A}} + \langle \delta(v), w \bullet_{\mathcal{A}} u \rangle_{\mathcal{A}} = -\langle \left( \delta\circ \Ll_u - \Rr_u^* \circ \delta\right)(v), w \rangle_{\mathcal{A}}.
\]
Since  $\prs_{\mathcal{A}}$ is non-degenerate, it follows that $\phi$ satisfies Condition~\eqref{line-central} if and only if $\delta \circ \Ll_u = \Rr_u^* \circ \delta$ for all $u \in \mathcal{A}$, as desired.
\end{proof}

\begin{Le}\label{Lemma2}
$(\h, \widetilde{\A} := \mathcal{A} \oplus \h^*, \widetilde{D}, \widetilde{\delta}, a_0, 0)$ is a context of an  almost semi-direct product of nearly associative algebras if and only if the following conditions hold for all $u, v \in \mathcal{A}$:
\begin{equation}\label{con-ex1}\begin{cases}
\de(u\bullet_{\mathcal{A}} v)=v\bullet_{\mathcal{A}} D(u),\; \de(u)\bullet_{\mathcal{A}} v=D(v\bullet_{\mathcal{A}} u),\; u\bullet_{\mathcal{A}} \de(v)=D(u)\bullet_{\mathcal{A}} v,\;\\   \delta\circ D = \delta^*\circ \delta = \Rr_{a_0}=-\delta^2 = -D^2,\;
\, D\circ \delta = \Ll_{a_0}=0,\\
 \delta(a_0)=\delta^*(a_0)=D(a_0)=0,\; \langle a_0, a_0\rangle_\A=0.
\end{cases}
\end{equation}
\end{Le}

\begin{proof}
Let $u, v \in \mathcal{A}$ and $\alpha, \beta \in \mathbb{K}$. A straightforward computation shows that the following equivalences hold under the product $\widetilde{\bullet}$ on $\widetilde{\A}$:
\begin{itemize}
    \item The condition $\widetilde{\delta}((u+\alpha e)\widetilde{\bullet} (v+\beta e)) = (v+\beta e) \widetilde{\bullet} \widetilde{D}(u+\alpha e)$ holds if and only if
    \[
    \delta(u \bullet_{\mathcal{A}} v) = v \bullet_{\mathcal{A}} D(u) \quad \text{and} \quad D \circ \delta = 0.
    \]

    \item The equation $\widetilde{\delta}(u+\alpha e)\widetilde{\bullet} (v+\beta e) = \widetilde{D}((v+\beta e)\widetilde{\bullet} (u+\alpha e))$ is equivalent to
    \[
    \delta(u) \bullet_{\mathcal{A}} v = D(v \bullet_{\mathcal{A}} u) \quad \text{and} \quad \delta^2 = -\Rr_{a_0}^*.
    \]

    \item The relation $(u+\alpha e)\widetilde{\bullet} \widetilde{\delta}(v+\beta e) = \widetilde{D}(u+\alpha e) \widetilde{\bullet} (v+\beta e)$ is equivalent to
    \[
    u \bullet_{\mathcal{A}} \delta(v) = D(u) \bullet_{\mathcal{A}} v \quad \text{and} \quad \delta^* \circ \delta = \delta \circ D.
    \]

    \item The identity $\widetilde{\delta}^2(u+\alpha e) = \widetilde{D}^2(u+\alpha e)$ is equivalent to
    \[
    \delta^2 = D^2 \quad \text{and} \quad D(a_0) = 0.
    \]

    \item The condition $\widetilde{\delta} \left( \widetilde{D}(u+\alpha e)\right) = (u+\alpha e) \widetilde{\bullet} a_0$ is equivalent to
    \[
    \delta \circ D = \Rr_{a_0} \quad \text{and} \quad \delta^*(a_0) = 0.
    \]

    \item The relation $\widetilde{D}\left(\widetilde{\delta}(u+\alpha e)\right) = a_0 \widetilde{\bullet} (u+\alpha e)$ is equivalent to
    \[
    D \circ \delta = \Ll_{a_0} \quad \text{and} \quad \delta(a_0) = 0.
    \]

    \item Finally, the equation $(\widetilde{D} - \widetilde{\delta})(a_0) = 0$ is equivalent to
    \[
    (D - \delta)(a_0) = 0 \quad \text{and} \quad \langle a_0, a_0 \rangle_{\mathcal{A}} = 0.
    \]
\end{itemize}
From these relations, we observe that $\delta^* \circ \delta = \delta \circ D = \Rr_{a_0}$. Since the operator $\delta^* \circ \delta$ is naturally self-adjoint, it follows immediately that $\Rr_{a_0}$ is symmetric with respect to $\prs_{\mathcal{A}}$. Furthermore, combining the identities, we deduce the chain of equalities $\delta^* \circ \delta = \delta \circ D = \Rr_{a_0} = -\delta^2 = -D^2$, which completes the proof of the lemma.
\end{proof}

We next introduce the concept of the  double extension  for anti-left-invariant pseudo-Euclidean nearly associative algebras.  
This construction, inspired by the classical double extension method of A. Medina and  P. Revoy~\cite{Medina-Revoy} for quadratic Lie algebras,  
provides an effective way to generate new anti-left-invariant structures from lower-dimensional ones.  
The precise formulation of this extension process is given in the following theorem.

\begin{theo}\label{construction-ex}
Let $(\A,\bullet_{\mathcal{A}},\prs_\A)$ be an anti-left-invariant pseudo-Euclidean nearly associative algebra, let $\h:=\K d$ be a one-dimensional vector space and let $\h^*:=\K e$ be its dual space. Consider two endomorphisms $\delta,D\in End(\A)$ and an element $a_0\in \A$ such that $D$ is skew-symmetric with respect to $\prs_\A$ and the following conditions are satisfied, for all $u, v\in \A$
\begin{equation}\label{con-ex}\begin{cases}
\de(u\bullet_{\mathcal{A}} v)=v\bullet_{\mathcal{A}} D(u),\; \de(u)\bullet_{\mathcal{A}} v=D(v\bullet_{\mathcal{A}} u),\; u\bullet_{\mathcal{A}} \de(v)=D(u)\bullet_{\mathcal{A}} v,\;\\   \delta\circ D = \delta^*\circ \delta = \Rr_{a_0}=-\delta^2 = -D^2,\;
\delta\circ \Ll_u = \Rr_u^*\circ \delta,\, D\circ \delta = \Ll_{a_0}=0,\\
 \delta(a_0)=\delta^*(a_0)=D(a_0)=0,\; \langle a_0, a_0\rangle_\A=0.
\end{cases}
\end{equation}
Then the vector space
\(
\overline{\A} := \h^* \oplus \A \oplus \h
\)
endowed with the bilinear product $\bullet$ given by
\begin{equation}
d\bullet d = a_0,\quad 
d\bullet u = D(u) - \langle a_0, u\rangle_\A e,\quad
u\bullet v = u\bullet_{\mathcal{A}} v - \langle \delta(u), v\rangle_\A e,\quad
u\bullet d = \delta(u),
\label{pro-anti}\end{equation}
for   $u,v\in \A$, is a nearly associative algebra.  

Moreover, the bilinear form $\prs : \overline{\A}\times \overline{\A}\to \K$ defined by
\[
 \prs|_{\A \times \A}=\prs_\A,\qquad \langle d,e\rangle=\langle e,d\rangle=1,
\]
and vanishing otherwise, is an anti-left-invariant pseudo-Euclidean scalar product  on $(\overline{\A},\bullet)$.
In this case, $(\overline{\A}, \bullet, \prs)$ is an anti-left-invariant pseudo-Euclidean nearly associative algebra.
\end{theo}

\begin{proof}
According to Lemma.~\ref{Lemma1}, the relation 
\(\delta \circ \Ll_u = \Rr_u^* \circ \delta\), for all \(u \in \A\), ensures that the bilinear map
\(\phi(u,v)=-\langle \delta(u),v\rangle_\A\), for all $u,v\in \A$, satisfies Condition~\eqref{line-central}.  
Hence, the product on \(\widetilde{\A} := \A \oplus \h^*\) given by
$$
(u+\alpha e) \widetilde{\bullet} (v+\beta e)=u\bullet_{\mathcal{A}} v+\phi(u,v)e
$$
defines a nearly associative algebra, i.e., an annulator extension of \((\A,\bullet_{\mathcal{A}})\) by \(\h^*\) by means of $\phi$.  

Now, since $(D, \delta, a_0)$ satisfies the system~\eqref{con-ex}, it follows from Lemma~\ref{Lemma2} that $(\h, \mathcal{H} := \mathcal{A} \oplus \h^*, \widetilde{D}, \widetilde{\delta}, a_0, 0)$ constitutes a context of an almost semi-direct product of nearly associative algebras. Thus, $(\overline{\A}, \bullet)$ is a nearly associative algebra. 

Furthermore, it is straightforward to check that the left multiplication operators of $(\overline{\A}, \bullet)$ are anti-symmetric with respect to $\prs$. In conclusion,  $(\overline{\A}, \bullet, \prs)$ is an anti-left-invariant pseudo-Euclidean nearly associative algebra, which completes the construction.
\end{proof}

\begin{Def}
The anti-left-invariant pseudo-Euclidean nearly associative algebra $(\overline{\A}, \bullet, \prs)$ constructed in Theorem \ref{construction-ex} is called the  \textit{double extension of the anti-left-invariant pseudo-Euclidean nearly associative algebra}  $(\A, \bullet_{\mathcal{A}}, \prs_\A)$ by means of $(\delta, D, a_0)$.
\end{Def}

\begin{pr}\label{reduit}
Let $(\A, \bullet, \prs)$ be an anti-left-invariant pseudo-Euclidean nearly associative algebra, and let $I := \K e \subset N(\A) \cap N(\A)^\perp$ be a non-zero vector subspace. 
We denote by $\mathcal{B} := I^\perp / I$ the quotient space and by 
$\pi : I^\perp \to \mathcal{B}$ the canonical projection. Then:
\begin{enumerate}
    \item $I$ is a totally isotropic two-sided ideal of $(\A,\bullet)$, and $I^\perp$ is also a two-sided ideal of $(\A,\bullet)$.
    
    \item The quotient $\mathcal{B} = I^\perp / I$ admits a canonical structure of anti-left-invariant pseudo-Euclidean nearly associative algebra \(
(\mathcal{B}  , \, \bullet_\mathcal{B}, \, \prs_\mathcal{B})
\), given by

\begin{equation}
 \pi(u) \bullet_{\mathcal{B}} \pi(v):=  \pi(u \bullet v), 
        \qquad 
        \langle \pi(u), \pi(v) \rangle_{\mathcal{B}} := \langle u, v \rangle,
\label{Bquocient}\end{equation}
for all $u,v \in I^\perp$.
\end{enumerate}
\end{pr}

\begin{proof}
\begin{enumerate}
\item  Since $I = \K e \subset N(\A) \cap N(\A)^\perp$, we have $I \bullet \A = \A \bullet I = \{0\}$. Thus, $I$ is a totally isotropic two-sided ideal of $(\A, \bullet)$. Now, let $v \in I^\perp$, $w \in \A$, and $u \in I$. By the anti-left-invariance of $\prs$ and since $  \A \bullet I = \{0\}$, we have
\[
\langle v \bullet w, u \rangle = -\langle w, v \bullet u \rangle = 0,
\qquad
\langle w \bullet v, u \rangle = -\langle v, w \bullet u \rangle = 0.
\]
Hence, $v \bullet w \in I^\perp$ and $w \bullet v \in I^\perp$, which proves that $I^\perp$ is a two-sided ideal of $(\A, \bullet)$.

\item Since $I$ and $I^\perp$ are two-sided ideals of $(\A, \bullet)$, it is straightforward to show that the quotient operations defined by
\[
\pi(u) \bullet_{\mathcal{B}} \pi(v)  = \pi(u \bullet v), 
\qquad 
\langle \pi(u), \pi(v) \rangle_{\mathcal{B}}  = \langle u, v \rangle,
\]
for all $u,v \in I^\perp$, endow the quotient space $\mathcal{B}  = I^\perp / I$ with a well-defined structure of an anti-left-invariant pseudo-Euclidean nearly associative algebra.

\end{enumerate}
\end{proof}

We now turn to the converse of Theorem \ref{construction-ex}, which provides a characterization of the algebras obtained by this construction.
\begin{theo}\label{Th-ex1}
Let $(\A, \bullet, \prs)$ be a non-trivial  anti-left-invariant pseudo-Euclidean nearly associative algebra.  
Then there exists a one-dimensional totally isotropic two-sided ideal $I \subset \A $ such that $(\A, \bullet, \prs)$ is a double extension of the anti-left-invariant pseudo-Euclidean nearly associative algebra 
\(
(\mathcal{B} := I^\perp / I, \, \bullet_\mathcal{B}, \, \prs_\mathcal{B})
\),
where $\bullet_\mathcal{B}$ and $\prs_\mathcal{B}$ are defined by \eqref{Bquocient}, 
by means of $(\delta, D, a_0)$.  Moreover, if  $(\A, \bullet, \prs)$ has signature $(q,n-q)$, then $(\mathcal{B}, \bullet_\mathcal{B}, \prs_\mathcal{B})$ has signature $(q-1,n-q-1)$. 
\end{theo}
\begin{proof}
According to Proposition~\ref{pr9}, we have $N(\A)\cap N(\A)^\perp \neq \{0\}$. Let $I := \K e$ be a non-zero vector subspace of $N(\A) \cap N(\A)^\perp$. Then, by Proposition~\ref{reduit}, $I$ is a totally isotropic two-sided ideal of $(\A,\bullet)$, as well as $I^\perp$ is  a two-sided ideal of $(\A,\bullet)$. Since $\A \bullet \A \neq \{0\}$, $\prs$ is non-degenerate and $e\neq 0$, there exists an element $d \in \A \setminus \{0\}$ such that $
\langle e,d\rangle = 1$ and $\langle d,d\rangle = 0.$

Let $V := \K d$ and $H := (I \oplus V)^\perp$. Then we have the decomposition
\(
\A = I \oplus H \oplus V,\) and \( I^\perp = I \oplus H.
\) For all $u,v \in H$, we have $u\bullet v\in I^\perp$ and we can write
\[
u \bullet v = u \bullet_H v + \varphi(u,v)e,
\]  
where $\varphi : H \times H \to \K$ is a bilinear form and $\bullet_H : H \times H \to H$ is a bilinear map.  
It is straightforward to verify that $(H, \bullet_H)$ is a nearly associative algebra and that the restriction
\(
\prs_H := \prs|_{H \times H}
\)
defines an anti-left-invariant pseudo-Euclidean  scalar product on $(H,\bullet_H)$.  
Therefore, $\varphi$ satisfies Condition~\eqref{line-central}.

According to Proposition~\ref{reduit}, the quotient algebra $(\mathcal{B} := I^\perp / I, \, \bullet_\mathcal{B}, \, \prs_\mathcal{B})$,
given by \eqref{Bquocient}, is an anti-left-invariant pseudo-Euclidean nearly associative algebra. The canonical projection $\pi: I^\perp \to \mathcal{B}$ allows us to identify $(H,\bullet_H,\prs_H)$ with $(\mathcal{B},\bullet_\mathcal{B},\prs_\mathcal{B})$. Thus, we can write $I^\perp \simeq I \oplus \mathcal{B}$. Since $I^\perp$ is a two-sided ideal of $(\A,\bullet)$, we obtain
\begin{equation*}
d \bullet d = a_0 + \alpha e + \lambda d, \quad 
d \bullet u = D(u) + f(u)e, \quad
u \bullet d = \delta(u) + g(u)e,
\label{pro-anti0}
\end{equation*}
for all $u \in \mathcal{B}$, where $D, \delta \in End(\mathcal{B})$, $a_0 \in \mathcal{B}$,  $f,g \in \mathcal{B}^*$, and   $\alpha, \lambda \in \K$.  

Since $\prs$ is an anti-left-invariant pseudo-Euclidean scalar product on $(\A, \bullet)$, using the identification of  $ H $ with $ \mathcal{B} $ one has for all $  u,v \in \mathcal{B}$ 
\[
\varphi(u,v) = \langle u \bullet v, d\rangle = - \langle v, u \bullet d \rangle = - \langle \delta(u), v\rangle_\mathcal{B}, \quad \text{for all } u,v \in \mathcal{B}.
\]
Hence,
\[
\langle D(u), v\rangle_\mathcal{B} = \langle d \bullet u, v\rangle = - \langle u, d \bullet v \rangle = - \langle u, D(v) \rangle_\mathcal{B}, \quad \text{for all } u,v \in \mathcal{B},
\]
which shows that $D$ is skew-symmetric with respect to $\prs_\mathcal{B}$.

Furthermore,  for all $u \in \mathcal{B}$ we obtain
\[
\langle a_0, u\rangle_\mathcal{B} = \langle d \bullet d, u\rangle = - \langle d, d \bullet u \rangle = - f(u),
\quad g(u) = \langle u \bullet d, d \rangle = 0,
\]
and
\[
\lambda = \langle d \bullet d, d\rangle = 0, 
\quad \alpha = \langle d \bullet d, e \rangle = - \langle d, d \bullet e \rangle = 0.
\]

Therefore, the product takes the simplified form
\begin{equation}
d \bullet d = a_0, \quad
d \bullet u = D(u) - \langle a_0, u \rangle_\mathcal{B} e, \quad
u \bullet v = u \bullet_\mathcal{B} v - \langle \delta(u), v \rangle_\mathcal{B} e, \quad
u \bullet d = \delta(u),
\label{pro-anti1}
\end{equation}
for all $u,v \in \mathcal{B}$.

Now, we show that the product \eqref{pro-anti1} defines a nearly associative algebra.  
For all $u,v \in \mathcal{B}$.

\begin{enumerate}
\item Since $\varphi(u,v) = -\langle \delta(u), v \rangle_\mathcal{B}$, then, according to Lemma~\ref{Lemma1}, we deduce that $\de\circ \Ll_u=\Rr_u^*\circ \de$, for all $u\in \mathcal{B}$.
\item We have 
    \begin{align*}
        (u \bullet v) \bullet d - v \bullet (d \bullet u) 
        &= (u \bullet_\mathcal{B} v - \langle \delta(u), v\rangle_\mathcal{B} e) \bullet d 
           - v \bullet (D(u) - \langle a_0, u \rangle_\mathcal{B} e) \\
           &= \de( u \bullet_\mathcal{B} v)-v\bullet_\mathcal{B} D(u)+\langle \de(v), D(u)\rangle_\mathcal{B} e\\
           &= \de( u \bullet_\mathcal{B} v)-v\bullet_\mathcal{B} D(u) -\langle D(\de(v)), u\rangle_\mathcal{B} e.
    \end{align*}
    Therefore, the first  vanishes if and only if
    \(
    \delta(u \bullet_\mathcal{B} v) = v \bullet_\mathcal{B} D(u)\) and  \(D\circ \de = 0.
    \)
\item We get
    \begin{align*}
        (u \bullet d) \bullet v - d \bullet (v \bullet u) 
        &= \delta(u)  \bullet v - d \bullet (v \bullet_\mathcal{B} u - \langle \delta(v), u\rangle_\mathcal{B} e) \\
        &= \delta(u) \bullet_\mathcal{B} v - \langle \delta^2(u), v\rangle_\mathcal{B} e - D(v \bullet_\mathcal{B} u) + \langle a_0, v \bullet_\mathcal{B} u \rangle_\mathcal{B} e \\
        &=\delta(u) \bullet_\mathcal{B} v - D(v \bullet_\mathcal{B} u) - \langle \delta^2(u), v\rangle_\mathcal{B} e - \langle v \bullet_\mathcal{B} a_0, u \rangle_\mathcal{B} e\\
        &=\delta(u) \bullet_\mathcal{B} v - D(v \bullet_\mathcal{B} u) - \langle \delta^2(u), v\rangle_\mathcal{B} e - \langle v, \Rr^*_{a_0}(u) \rangle_\mathcal{B} e.
    \end{align*}
    This vanishes if and only if  $\delta^2 = -\Rr^*_{a_0}$ and $\delta(u) \bullet_\mathcal{B} v = D(v \bullet_\mathcal{B} u )$.
   \item We obtain 
\begin{align*}
(u \bullet d) \bullet d - d \bullet (d \bullet u) 
&=  \delta(u) \bullet d  - d \bullet (D(u) - \langle a_0, u \rangle_\mathcal{B} e) \\
&= \delta^2(u) - D^2(u) + \langle a_0, D(u) \rangle_\mathcal{B} e.
\end{align*}
Therefore, this  vanishes if and only if
\(
\delta^2 = D^2\) and  \(  D(a_0) = 0\).
  \black

 \item We have 
\begin{align*}
(d \bullet u) \bullet v - u \bullet (v \bullet d) 
&= (D(u) - \langle a_0, u\rangle_\mathcal{B} e) \bullet v - u \bullet  \delta(v)  \\
&= D(u) \bullet_\mathcal{B} v - \langle \delta(D(u)), v \rangle_\mathcal{B} e - u \bullet_\mathcal{B} \delta(v) + \langle \delta(u), \delta(v) \rangle_\mathcal{B} e \\
&= D(u) \bullet_\mathcal{B} v - u \bullet_\mathcal{B} \delta(v) - \langle \delta \circ D(u) - \delta^*\circ\delta(u), v \rangle_\mathcal{B} e.
\end{align*}
Hence, this vanishes if and only if 
\(
D(u) \bullet_\mathcal{B} v = u \bullet_\mathcal{B} \delta(v), \) and \( \delta \circ D =\delta^*\circ\delta.
\)

\item We have 
\begin{align*}
(d \bullet u) \bullet d - u \bullet (d \bullet d) 
&= (D(u) - \langle a_0, u\rangle_\mathcal{B} e) \bullet d - u \bullet a_0 \\
&= \delta(D(u)) - u \bullet_\mathcal{B} a_0 + \langle \delta(u), a_0 \rangle_\mathcal{B} e \\
&= \delta(D(u)) - u \bullet_\mathcal{B} a_0 + \langle \delta^*(a_0), u\rangle_\mathcal{B} e.
\end{align*}
This expression vanishes if and only if the following conditions hold: $\delta\circ D= \Rr_{a_0} $ and $\delta^*(a_0)=0$.
\item We get
\begin{align*}
(d \bullet d) \bullet u - d \bullet (u \bullet d) &= a_0\bullet u-d\bullet \de(u)\\&= a_0\bullet_\mathcal{B} u-\langle \de(a_0), u\rangle e-D(\de(u))+\langle a_0, \de(u)\rangle_\mathcal{B} e.\end{align*}
So, this vanishes if and only if $(\de-\de)^*(a_0)=0$ and $D\circ \de=\Ll_{a_0}$.
\item We have 
\begin{align*}
(d \bullet d) \bullet d - d \bullet (d \bullet d) &=a_0\bullet d-d\bullet a_0\\&=\de(a_0)-D(a_0)+\langle a_0, a_0\rangle_\mathcal{B} e.\end{align*}
Thus, this vanishes if and only if 
$(\de-D)(a_0)=0$ and $\langle a_0, a_0\rangle_\mathcal{B}=0.$
\end{enumerate}
Therefore, $(\A,\bullet)$ is nearly associative if and only if the system of relations~\eqref{con-ex} is satisfied.  
Consequently, $(\A,\bullet,\prs)$ is a double extension of  an anti-left-invariant pseudo-Euclidean nearly associative algebra $(\mathcal{B}, \bullet_\mathcal{B}, \prs_\mathcal{B})$ by means of $(\de, D, a_0)$. Moreover, by straightforward calculations, it is clear that if $(\A,\bullet,\prs)$ has signature $(q, n-q)$, then $(\mathcal{B}, \bullet_\mathcal{B}, \prs_\mathcal{B})$ has signature $(q-1,n-q-1)$.
\end{proof}

\begin{co}\label{Lor}
   Let $(\A, \bullet, \prs)$ be an anti-left-invariant Lorentzian nearly associative algebra.  
   Then $(\A, \bullet)$ is a trivial algebra. 
\end{co}

\begin{proof}
Assume, by contradiction, that $(\A, \bullet)$ is non-trivial, i.e., $\A \bullet \A \neq \{0\}$.  
Then, according to Theorem~\ref{Th-ex1}, $(\A,\bullet,\prs)$ is a double extension of an anti-left-invariant Euclidean nearly associative algebra $(\mathcal{B}, \bullet_\mathcal{B}, \prs_\mathcal{B})$ by means of $(\delta, D, a_0)$.   By Proposition~\ref{pr9},  $(\mathcal{B}, \bullet_\mathcal{B})$ is trivial algebra.   Hence, according to the system of relations~\eqref{con-ex}, we have
\[
D^2 = \delta^2 = \delta \circ D = D \circ \delta = \delta^* \circ \delta = 0, 
\quad \langle a_0, a_0 \rangle_\mathcal{B} = 0,
\]
and $D$ is skew-symmetric with respect to $\prs_\mathcal{B}$.   Since $(\mathcal{B}, \prs_\mathcal{B})$ is a Euclidean space, its scalar product is positive definite. Thus, the relations $\langle a_0, a_0 \rangle_\mathcal{B} = 0$ and $\delta^* \circ \delta = 0$ immediately imply $a_0 = 0$ and $\delta = 0$. Furthermore, since $D$ is skew-symmetric and satisfies $D^2 = 0$, it follows that $D = 0$.

Finally, by the structure of the product in~\eqref{pro-anti}, we deduce that $(\A, \bullet)$ is a trivial algebra, which contradicts our initial assumption. Hence, $(\A, \bullet)$ must be trivial.
\end{proof}

\begin{co}
   Let $(\A, \bullet, \prs)$ be an anti-left-invariant pseudo-Euclidean nearly associative algebra of signature $(2, n-2)$ with $n \geq 4$.  
   Then $(\A, \bullet)$ is nilpotent of index at most $3$. Moreover, the algebra $(\A, \bullet)$ is associative.
\end{co}

\begin{proof}
Suppose that $(\A, \bullet)$ is non-trivial.  
Then, by Theorem~\ref{Th-ex1}, $(\A,\bullet,\prs)$ is a double extension of an anti-left-invariant Lorentzian nearly associative algebra $(\mathcal{B}, \bullet_\mathcal{B}, \prs_\mathcal{B})$ by means of $(\delta, D, a_0)$.  According to Corollary~\ref{Lor}, $(\mathcal{B}, \bullet_\mathcal{B})$ is trivial algebra.  
Consequently, from the system of relations~\eqref{con-ex}, we deduce that
\[
D^2 = \delta^2 = \delta \circ D = D \circ \delta = \delta^* \circ \delta = 0, 
\quad \langle a_0, a_0 \rangle_\mathcal{B} = 0,
\]
and that $D$ is skew-symmetric with respect to $\prs_\mathcal{B}$.  
Hence, by the expression of the product in~\eqref{pro-anti}, one easily checks that 
\[
(u \bullet v) \bullet w = 0, \qquad \text{for all }\, u,v,w \in \A.
\]
Since $\A$ satisfies Relation \eqref{nearly-associative}, this shows that the algebra $(\A, \bullet)$ is nilpotent of index at most $3$. Therefore, the algebra $(\A, \bullet)$ is associative. 
\end{proof}

\begin{co} A non-trivial anti-left-invariant pseudo-Euclidean nearly associative  algebra can be obtained by a sequence of  double extensions of a trivial pseudo-Euclidean  algebra. 
\end{co}
\begin{proof}
 According to Theorem~\ref{Th-ex1},  a non-trivial  anti-left-invariant pseudo-Euclidean nearly associative algebra $(\A, \bullet, \prs)$ is a double extension of an anti-left-invariant pseudo-Euclidean nearly associative algebra $(\mathcal{B}_1, \bullet_{\mathcal{B}_1}, \prs_{\mathcal{B}_1})$. If $(\mathcal{B}_1, \bullet_{\mathcal{B}_1})$ is trivial, the proof is complete. Otherwise, $(\mathcal{B}_1, \bullet_{\mathcal{B}_1}, \prs_{\mathcal{B}_1})$ can itself be obtained as a double extension of an anti-left-invariant pseudo-Euclidean nearly associative algebra $(\mathcal{B}_2, \bullet_{\mathcal{B}_2}, \prs_{\mathcal{B}_2})$.

Since every double extension strictly reduces the dimension of the underlying algebra, this inductive process must terminate after a finite number of steps. Consequently, there exists a positive integer $k$ such that $(\mathcal{B}_k, \bullet_{\mathcal{B}_k})$ is a trivial algebra, which completes the proof.
\end{proof}

\section{Classification of anti-left-invariant  pseudo-Euclidean nearly associative algebras of dimension $\leq 6$ with signature $(2,n-2)$}\label{s:Class anti-left}

In this section, we apply Theorem~\ref{Th-ex1} to classify all anti-left-invariant  pseudo-Euclidean nearly associative algebras of dimension $\leq 6$ with $\prs$ of signature $(2,n-2)$.

According to Theorem~\ref{Th-ex1}, every anti-left-invariant  pseudo-Euclidean nearly associative algebra of signature $(2,n-2)$ is obtained from an anti-left-invariant  Lorentzian nearly associative algebra
$ (\mathcal{B},\bullet_{\mathcal{B}},\langle \cdot,\cdot \rangle_{\mathcal{B}})$
by a double extension determined by means of $(D, \delta, a_0)$. Then, Corollary~\ref{Lor} implies that $(\mathcal{B},\bullet_{\mathcal{B}})$ is necessarily a trivial algebra.
Consequently, the equations \eqref{con-ex} reduce to
\begin{equation}
D^2=\delta^2=D\circ\de=\de\circ D=\delta^*\circ \delta=\delta(a_0)=\delta^*(a_0)=0,
\qquad
\langle a_0,a_0\rangle_{\mathcal{B}}=0,
\label{trivial0}
\end{equation}
and $D$ is skew-symmetric with respect to $\langle \cdot,\cdot \rangle_{\mathcal{B}}$.

By a standard argument based on the signature of the bilinear form, we obtain the following lemma.
\begin{Le}\label{Dzero}
Let $E$ be a vector space endowed with a non-degenerate symmetric bilinear form $\langle \cdot,\cdot\rangle$ of signature $(1,n-1)$, and let $D \in \mathrm{End}(E)$. If
\(
D^2=0
\) and $D$ is skew-symmetric with respect to $\langle \cdot,\cdot\rangle$, then
\(
D=0.
\)
\end{Le}

\begin{pr}\label{SoluLor}
Let $(\mathcal{B},\prs)$ be a Lorentzian vector space. Then
\begin{enumerate}
\item If $\dim \mathcal{B}=2$, then $(D,\delta,a_0)$ satisfy \eqref{trivial0} if and only if there exists a basis $\mathbb{B}=\{e_1,e_2\}$ of $\mathcal{B}$  such that one of the following holds:

\begin{itemize}
\item[$(a_1)$] $D=\delta=0$,  $a_0=\lambda e_1$, where $\lambda\in\mathbb{R}$. Moreover, the metric is given by
$
\prs = e_1^{*}\odot e_2^{*}.
$

\item[$(a_2)$] $D=0$, 
$\mathrm{Mat}(\delta,\mathbb{B})=
\begin{pmatrix}
0 & x\\
0 & 0
\end{pmatrix}$,
 $a_0=\lambda e_1$,
where $x,\lambda\in\mathbb{R}$ with $x\neq 0$. Moreover, the metric is
$
\prs = e_1^{*}\odot e_2^{*}.
$
\end{itemize}

\item If $\dim \mathcal{B}=3$, then $(D,\delta,a_0)$ satisfy \eqref{trivial0} if and only if there exists a basis $\mathbb{B}=\{e_1,e_2,e_3\}$ of $\mathcal{B}$ such that one of the following holds:

\begin{itemize}
\item[$(a_3)$] $D=\delta=0$,  $a_0=\lambda e_1$, where $\lambda\in\mathbb{R}$. Moreover, the metric is
$
\prs = e_1^{*}\odot e_2^{*} + e_3^{*}\otimes e_3^{*}.
$
\item[$(a_4)$] $D=0$,
$\mathrm{Mat}(\delta,\mathbb{B})=
\begin{pmatrix}
0 & x & y\\
0 & 0 & 0\\
0 & 0 & 0
\end{pmatrix}$, $a_0=\lambda e_1$,
where $x, y, \lambda\in\mathbb{R}$ and $(x,y)\neq(0,0)$. Moreover,
$
\prs = e_1^{*}\odot e_2^{*} + e_3^{*}\otimes e_3^{*}.
$
\end{itemize}

\item If $\dim \mathcal{B}=4$, then $(D,\delta,a_0)$ satisfy \eqref{trivial0} if and only if there exists a basis $\mathbb{B}=\{e_1,e_2,e_3,e_4\}$ of $\mathcal{B}$ such that one of the following holds:

\begin{itemize}
\item[$(a_5)$] $D=\delta=0$, $a_0=\lambda e_1$, where $\lambda\in\mathbb{R}$. Moreover,
$
\prs =
e_1^{*}\odot e_2^{*} + e_3^{*}\otimes e_3^{*} + e_4^{*}\otimes e_4^{*}.
$

\item[$(a_6)$] $D=0$, 
$\mathrm{Mat}(\delta,\mathbb{B})=
\begin{pmatrix}
0 & x & y & z\\
0 & 0 & 0 & 0\\
0 & 0 & 0 & 0\\
0 & 0 & 0 & 0
\end{pmatrix}$, $a_0=\lambda e_1$,
where $x,y,z,\lambda\in\mathbb{R}$ and $(x,y,z)\neq(0,0,0)$. Moreover,
$
\prs =
e_1^{*}\odot e_2^{*} + e_3^{*}\otimes e_3^{*} + e_4^{*}\otimes e_4^{*}.
$
\end{itemize}
\end{enumerate}

\end{pr}

\begin{proof}
 First, note that according to \eqref{trivial0}, we have $D^2=0$ and $D$ is skew-symmetric with respect to $\prs$. Hence, by Lemma~\ref{Dzero}, we obtain $D=0$. Consequently, the Conditions \eqref{trivial0} reduce to
\begin{equation}\label{eq: trivial D=0}
  D=\delta^2=\delta^*\circ \delta=\delta(a_0)=\delta^*(a_0)=0,
\qquad
\langle a_0,a_0\rangle=0.  
\end{equation}

 \begin{enumerate}
        \item Assume that $\dim \mathcal{B}=2$. Let $(D,\delta,a_0)$ satisfy
\eqref{eq: trivial D=0}.
            If $\delta=0$, then $\langle a_0, a_0\rangle=0$ implies that there exists a  basis $\{e_1, e_2\}$ of $ \mathcal{B}$ such that $a_0=\la e_1$, where $\la\in \mathbb{R}$, and the metric $\prs =e_1^{*}\odot e_2^{*}$. 
        
       Suppose now that $\de\neq 0$. Since $\de^*\circ \de=0$, it follows that $\mathrm{Im} (\de )$ is a totally isotropic subspace. Since every totally isotropic subspace of a Lorentzian space has dimension at most 1 and $\de\neq 0$, it follows that $\dim \mathrm{Im} (\de )=1$.  Therefore, there exist an isotropic vector $e_1\in\mathcal{B}$ and a covector $\al\in \mathcal{B}^*$ such that $\de(a) = \al(a)e_1$ for any $ a\in \mathcal{B}$. 
       
       Since  $(\mathcal{B},\prs)$ is a two-dimensional Lorentzian space, we can choose an isotropic vector  $e_2\in \mathcal{B}$ such that $\langle e_1, e_2\rangle=1$. Then $\{e_1, e_2\}$ is a basis of $\mathcal{B}$ and $\prs = e_1^{*}\odot e_2^{*}$. 
       Since $\de^2=0$, it follows that $0=\de^2(e_1)=\alpha(e_1)^2 e_1$, and thus $\de(e_1)=0$. Moreover, since $\delta\neq 0$, we have  $\alpha(e_2)\neq 0$, and then $\de(e_2)=x e_1$ where $x:=\alpha(e_2) \neq 0$.
        Finally, since $\delta(a_0)=0$, we obtain $\alpha(a_0)=0$. As $\alpha(e_2)=x\neq0$, it follows that $a_0=\lambda e_1$ for some $\lambda\in\mathbb{R}$.

Conversely, a direct verification shows that if there exists a basis $\mathbb{B}$ of $\mathcal{B}$ that satisfies either $(a_1)$ or $(a_2)$, then $(D,\delta,a_0)$ satisfies \eqref{trivial0}.

\smallskip

\item  Assume that $\dim \mathcal{B}=3$.  If $\de=0$, then $\langle a_0, a_0\rangle=0$ implies that there exists a  basis $\{e_1, e_2, e_3\}$ of $ \mathcal{B}$ such that $a_0=\la e_1$ where $\la\in \mathbb{R}$ and the metric $\prs =e_1^{*}\odot e_2^{*}+e_3^{*}\otimes e_3^{*}$. 

Suppose now $\de\neq 0$. Since $\de^*\circ \de=0$, it follows that $\mathrm{Im} (\de)$ is a totally isotropic subspace, and hence there exist an isotropic vector $e_1\in \mathcal{B}$
and a covector $\al\in \mathcal{B}^*$ such that $\de(a) = \al(a)e_1$ for any $ a\in \mathcal{B}$. 

Since  $(\mathcal{B},\prs)$ is a three-dimensional Lorentzian space, we can choose a basis $\{e_1,e_2,e_3\}$ of $\mathcal{B}$ such that $\spa\{e_1,e_2\}$ and $\spa\{e_3\}$ are
orthogonal, $\langle e_3, e_3\rangle=1$  and $\langle e_2, e_2\rangle=0$. 
Since $\de^2=0$, we have $\de(e_1)=0$. Moreover,  as $\delta\neq 0$, it follows that $\left(\alpha(e_2),\alpha(e_3)\right)\neq (0,0)$, and so
$\de(e_2)=x e_1$, $\de(e_3)=y e_1$, where $x:=\alpha(e_2), y:= \alpha(e_3)\in \mathbb{R}$ and    $(x,y)\neq(0,0)$. 
Finally, let $a_0=\la e_1+\la_2 e_2+\la_3 e_3$, with $\la,\la_2,\la_3\in\mathbb{R}$. As $\de(a_0)=\de^*(a_0)=0$, then we have $\la_2 x+\la_3 y=0$ and $x\la_2=y\la_2=0$. Since $(x,y)\neq (0,0)$, it follows that $\la_2=0$. Moreover, $\langle a_0,a_0\rangle=0$ implies that $\la_3=0$. Hence,  $a_0=\la e_1$, where $\la \in \mathbb{R}.$

Conversely, a direct verification shows that if there exists a basis $\mathbb{B}$ of $\mathcal{B}$ that satisfies either $(a_3)$ or $(a_4)$, then $(D,\delta,a_0)$ satisfies \eqref{trivial0}.

\smallskip

\item  Assume that $\dim \mathcal{B}=4$.  If $\de=0$, then $\langle a_0, a_0\rangle=0$ implies that there exists a  basis $\{e_1, e_2, e_3, e_4\}$ of $ \mathcal{B}$ such that $a_0=\la e_1$ where $\la\in \mathbb{R}$ and the metric is $\prs =e_1^{*}\odot e_2^{*}+e_3^{*}\otimes e_3^{*}+e_4^{*}\otimes e_4^{*}.$ 

Now suppose $\de\neq 0$. Since $\de^*\circ \de=0$, it follows that $\mathrm{Im} (\de)$ is a totally isotropic subspace and hence there exist an isotropic vector $e_1\in \mathcal{B}$
and a covector $\al\in \mathcal{B}^*$ such that $\de(a) = \al(a)e_1$ for any $ a\in \mathcal{B}$. 

Since  $(\mathcal{B},\prs)$ is a four-dimensional Lorentzian space, we can choose basis $\{e_1,e_2,e_3,e_4\}$ of $\mathcal{B}$ such that $\operatorname{span}\{e_1,e_2\}$ and $\operatorname{span}\{e_3,e_4\}$  are orthogonal to each other, where $e_2$ is an isotropic vector satisfying $\langle e_1, e_2 \rangle = 1$, and $\operatorname{span}\{e_3, e_4\}$ is an orthonormal subspace.  Since $\de^2=0$, we have $\de(e_1)=0$. Moreover,  as $\delta\neq 0$,  we have $\de(e_2)=x e_1$, $\de(e_3)=y e_1$ and $\de(e_4)=z e_1$, where $x:=\alpha(e_2),\, y:=\alpha(e_3),\, z:=\alpha(e_4)\in \mathbb{R}$ such that $(x,y,z)\neq (0,0,0)$. 
Finally, let $a_0=\la  e_1+\la_2 e_2+\la_3 e_3+\la_4 e_4$, with $\la ,\la_2,\la_3,\la_4
\in\mathbb{R}$. Since $\de(a_0)=\de^*(a_0)=0$, we have $x\la_2+y\la_3+z\la_4=0$ and $x\la_2=y\la_2=z\la_2=0$. As $(x,y,z)\neq (0,0,0)$, it follows that $\la_2=0$. Moreover, $\langle a_0,a_0\rangle=0$ implies that $\la_3=\la_4=0$. Thus,  $a_0=\la  e_1$, where $\la  \in \mathbb{R}.$

Conversely, a direct verification shows that if there exists a basis $\mathbb{B}$ of $\mathcal{B}$ that satisfies either $(a_5)$ or $(a_6)$, then $(D,\delta,a_0)$ satisfies \eqref{trivial0}.
\end{enumerate}
\end{proof}

We now have all the necessary ingredients to present the list of  anti-left-invariant pseudo-Euclidean nearly associative  algebras of dimension $\le 6$ where $\prs$ has signature $(2,n-2)$. 
\begin{theo}
   \begin{enumerate}
       \item Every  4-dimensional  anti-left-invariant pseudo-Euclidean nearly associative  algebra where $\prs$ of signature $(2,n-2)$ is isometric to one of the following pseudo-Euclidean algebras:
       
    \begin{enumerate}
           \item $\A_{4,1}=(\spa(d,e_1,e_2,e),\bullet,\prs)$, where  the only non-trivial products among the generators are given by 
       $$d\bullet d=\la e_1,\qquad d\bullet e_2=-\la e,
       $$
      with $\la\in \mathbb{R}$, and  the anti-left-invariant scalar product is defined by 
       $$
       \prs=d^*\odot e^*+e_1^*\odot e_2^*.$$

       \item  $\A_{4,2}=(\mathrm{span}(d,e_1,e_2,e),\bullet,\prs)$, where  the only non-trivial products among the generators are given by  
       $$
       d\bullet d=\la e_1,\quad d\bullet e_2=-\la e,\quad  e_2\bullet e_2=-xe,\quad e_2\bullet d=xe_1,
       $$
       with $ \la\in \mathbb{R}$, $x\in \mathbb{R}^*$, and  the anti-left-invariant scalar product is defined by 
       $$
       \prs=d^*\odot e^*+e_1^*\odot e_2^*. $$
        
       \end{enumerate}
       
        \item Every  5-dimensional  anti-left-invariant pseudo-Euclidean nearly associative  algebra where $\prs$ of signature $(2,n-2)$ is isometric to one of the following pseudo-Euclidean algebras:
        
        \begin{enumerate}
\item $\A_{5,1}=(\mathrm{span}(d,e_1,e_2, e_3,e),\bullet,\prs)$, where the only non-trivial products among the generators are given by 
$$
d\bullet d=\la e_1,\quad d\bullet e_2=-\la e,
$$
with $\la\in \mathbb{R}$, and   the anti-left-invariant scalar product is defined by 
$$
\prs=d^*\odot e^*+e_1^*\odot e_2^*+e_3^*\otimes e_3^*.$$

\item $\A_{5,2}=(\mathrm{span}(d,e_1,e_2, e_3,e),\bullet,\prs)$, where  the only non-trivial products among the generators are given by 
       \begin{align*}
           &d\bullet d=\la e,\ d\bullet e_2=-\la e,\ e_2\bullet e_2=-xe,\ e_3\bullet e_2=-y e,\ e_2\bullet d=x e_1,\ e_3\bullet d=y e_1
       \end{align*}
       with $x,y,\la\in \mathbb{R}$ such that $(x,y)\neq (0,0)$,  and  the anti-left-invariant scalar product is defined by 
       $$
       \prs=d^*\odot  e^*+e_1^*\odot e_2^*+e_3^*\otimes e_3^*. 
       $$
        \end{enumerate}

\item Every  6-dimensional  anti-left-invariant pseudo-Euclidean nearly associative  algebra where $\prs$ of signature $(2,n-2)$ is isometric to one of the following pseudo-Euclidean algebras:
        \begin{enumerate}
\item $\A_{6,1}=(\mathrm{span}(d,e_1,e_2, e_3, e_4, e),\bullet,\prs)$, where  the only non-trivial products among the generators are given by  
$$
d\bullet d=\la e_1,\quad d\bullet e_2=-\la e,
$$
with $\la\in \mathbb{R}$, and  the anti-left-invariant scalar product is defined by 
$$\prs=d^*\odot e^*+e_1^*\odot e_2^*+e_3^*\otimes e_3^*+e_4^*\otimes e_4^*.$$

\item $\A_{6,2}=(\mathrm{span}(d,e_1,e_2, e_3, e_4,e),\bullet,\prs)$, where  the only non-trivial products among the generators are given by  
       \begin{align*}
           &d\bullet d=\la e,\quad  d\bullet e_2=-\la e,\quad  e_2\bullet e_2=-xe,\quad  e_3\bullet e_2=-y e,\\
           & e_4\bullet e_2=-ze,\quad  e_2\bullet d=x e_1, \quad  
             e_3\bullet d=y e_1,\quad e_4\bullet d=z e_1,  \end{align*}
           with  $x,y,z,\la\in \mathbb{R}$ such that $(x,y, z)\neq (0,0,0)$, and  the anti-left-invariant scalar product is defined by 
           $$
           \prs=d^*\odot e^*+e_1^*\odot e_2^*+e_3^*\otimes e_3^*++e_4^*\otimes e_4^*.
           $$
       \end{enumerate}
   \end{enumerate}
\end{theo}

\begin{proof}
   According to Theorem~\ref{Th-ex1}, $(\A,\bullet,\prs)$ is a double extension of an anti-left-invariant pseudo-Euclidean nearly associative algebra
$
(\mathcal{B}, \bullet_{\mathcal{B}}, \prs_{\mathcal{B}})
$
by means of the data $(D,\delta,a_0)$.
By Corollary~\ref{Lor}, $(\mathcal{B},\bullet_{\mathcal{B}})$ is the trivial algebra. Hence, applying  Proposition~\ref{SoluLor}, all solutions of the defining equations are obtained. Therefore, from the product given in \eqref{pro-anti}, we obtain the corresponding algebras. 
\end{proof}

\section{Left-invariant pseudo-Euclidean nearly associative algebras}\label{s5}

In the present section, we address left-invariant  pseudo-Euclidean nearly associative algebras and study their fundamental properties.  
We show that the Lie algebra associated with any left-invariant  pseudo-Euclidean nearly associative algebra is solvable of index $2$.  
Moreover, we prove that every left-invariant Euclidean nearly associative algebra is a quadratic commutative associative algebra.  
We also establish that the Lie algebra associated with such an algebra is a cyclic pseudo-Euclidean Lie algebra.  
Since cyclic pseudo-Euclidean Lie algebras have been studied in \cite{Abid,ga}, the Lie algebras arising from left-invariant pseudo-Euclidean nearly associative algebras constitute a significant subclass of this well-studied family.

The problem under consideration admits two equivalent formulations, which we describe below.

\medskip
\noindent \textbf{Approach 1.}
Let $(\mathcal{J}, \In, \prs)$ be a pseudo-Euclidean Jordan algebra and let $\bullet$ be its Levi-Civita product. We investigate those pseudo-Euclidean Jordan algebras for which $(\mathcal{J}, \bullet)$ is a nearly associative algebra. Observe that the Levi-Civita product is automatically symmetric with respect to, namely,
\[
\langle u\bullet v,w\rangle = \langle v, u\bullet w\rangle,
\qquad \text{for all } u,v,w\in \mathcal{J}.
\]
Thus, the condition that $(\mathcal{J},\bullet)$ is a nearly associative algebra is equivalent to requiring that $(\mathcal{J},\bullet, \prs)$ be a left-invariant pseudo-Euclidean nearly associative algebra.

\medskip
\noindent \textbf{Approach 2 (Conversely).}
Let $(\A,\bullet, \prs)$ be a left-invariant pseudo-Euclidean nearly associative algebra, which in particular means that $\prs$ satisfies
\begin{equation} \label{lcs1}
\langle u\bullet v,w\rangle = \langle v, u\bullet w\rangle,
\qquad \text{for all } u,v,w\in \A. 
\end{equation}
By Proposition \ref{Lie-Jordan}, $\A^+=(\A, \{\cdot, \cdot\})$ is a Jordan algebra.
Hence, $(\A^+, \prs )$ is a pseudo-Euclidean Jordan algebra satisfying Condition \eqref{lcs1}. Therefore, by the uniqueness of the Levi-Civita product, it follows that the product $\bullet$ coincides with the Levi-Civita product associated with $(\A^{+},\prs)$.

\medskip
\noindent \textbf{Conclusion.}
The study of pseudo-Euclidean Jordan algebras whose Levi-Civita product is nearly associative is equivalent to the study of left-invariant pseudo-Euclidean nearly associative algebras.

\begin{remark}
Let $(\A, \bullet, \prs)$ be a left-invariant nearly associative algebra. If $(\A, \bullet)$ is commutative, then it is a quadratic commutative associative algebra. 
Quadratic nearly associative algebras, including the commutative associative case, have been studied in \cite{BenBar}. Therefore, in this paper, we focus on the study of non-commutative nearly associative algebras endowed with a left-invariant pseudo-Euclidean scalar product.
\end{remark}

\begin{Le}\label{Le2}
Let $(\A, \bullet, \prs)$ be a left-invariant pseudo-Euclidean nearly associative algebra. Then, for all $u, v, w \in \A$,  
\begin{equation} \label{eq:anti-left-lemma1}
(u \bullet v) \bullet w = v \bullet (w \bullet u) =  (v \bullet u) \bullet w =  u \bullet (w \bullet v).
\end{equation}
Moreover, the Lie algebra $\A^-$ is solvable of index  at most  $2$.
\end{Le}

\begin{proof}
Let $u, v, w, z \in \A$. Applying the nearly associativity identity we have
$(u \bullet v) \bullet w = v \bullet (w \bullet u) $ and  $(v \bullet u) \bullet w = u \bullet (w \bullet v)$.
Now, using the left-invariance of $\prs$ and the nearly associativity identity, we have:
\begin{align*}
\langle (u \bullet v) \bullet w, z \rangle 
= \langle w, (u \bullet v) \bullet z \rangle 
&\overset{\eqref{nearly-associative}}{=} \langle w, v \bullet (z \bullet u) \rangle 
= \langle z \bullet (v \bullet w), u \rangle \\
&\overset{\eqref{nearly-associative}}{=} \langle (w \bullet z) \bullet v, u \rangle = \langle v, (w \bullet z) \bullet u \rangle \\
&\overset{\eqref{nearly-associative}}{=} \langle v, z \bullet (u \bullet w) \rangle = \langle u \bullet (z \bullet v), w \rangle \\
&\overset{\eqref{nearly-associative}}{=} \langle (v \bullet u) \bullet z, w \rangle = \langle z, (v \bullet u) \bullet w \rangle.
\end{align*}
Since $\prs$ is non-degenerate, it follows that
\[
(u \bullet v) \bullet w = (v \bullet u) \bullet w.
\]
Thus, we obtain the desired identity~\eqref{eq:anti-left-lemma1}. Furthermore, since $(u \bullet v) \bullet w = (v \bullet u) \bullet w$, it immediately follows that $[u,v] \bullet w = 0$ for all $u,v,w \in \A$. This implies that $\A^-$ is solvable of index  at most  $2$.
\end{proof}

Let $(\A, \bullet, \prs)$ be a left-invariant pseudo-Euclidean nearly associative algebra.
It is straightforward to see that
\begin{equation}\label{ort-derive}
[\A, \A]^\perp = \{ u \in \A \mid \Rr_u^* = \Rr_u \}.
\end{equation}

\begin{pr}\label{Non-commutative}
Let $(\A, \bullet, \prs)$ be a  left-invariant pseudo-Euclidean non-commutative nearly associative  algebra. Then $[ \A, \A ] $ is degenerate.
\end{pr}
\begin{proof}
Suppose by contradiction that $[ \A, \A ]$ is nondegenerate. Then we have the orthogonal decomposition:
$
\A = [ \A, \A ] \oplus [ \A, \A ]^\perp.$ By Lemma~\ref{Le2}, $[v,w] \bullet z = 0$ for all $v,w,z \in \A$, which implies $\Ll_u = 0$ for all $u \in [ \A, \A ]$. Thus, for all $u \in [ \A, \A ]$ and $v, w \in [ \A, \A ]^\perp$, we have
$$
\langle v \bullet w, u \rangle \overset{\eqref{ort-derive}}{=} \langle v, u \bullet w \rangle = 0,
$$
which implies that $v \bullet w \in [ \A, \A ]^\perp$ and hence $[v,w]\in [\A, \A]^\bot\cap [\A, \A]$. Therefore, the bracket vanishes identically on $[ \A, \A ]^\perp$.

Since both $[\A, \A]$ and $[\A, \A]^\perp$ are abelian, the expansion of the bracket via the orthogonal decomposition reduces to:
$$[\A, \A] = [[\A, \A] \oplus [\A, \A]^\perp, [\A, \A] \oplus [\A, \A]^\perp] = [[\A, \A], [\A, \A]^\perp] + [[\A, \A]^\perp, [\A, \A]].$$
Given that $[u, a] = u \bullet a$ for all $u \in [\A, \A]^\perp$ and $a \in [\A, \A]$, it follows that:
$$[\A, \A] = [\A, \A]^\perp \bullet [\A, \A].$$
It follows that, for any $v \in [ \A, \A ]$, there exist $u \in [ \A, \A ]^\perp$ and $a \in [ \A, \A ]$ such that $v = u \bullet a$. For any $w \in [ \A, \A ]^\perp$,   since $\Ll_a = 0$, the nearly associative identity yields 
$$
w\bullet v=w \bullet (u \bullet a) = (a \bullet w) \bullet u = 0,
$$
Thus, $w \bullet v = 0$ for all $w \in [\A,\A]^\perp$. We deduce that $\Rr_v = \Ll_v = 0$, for all $v \in [ \A, \A ]$.

Finally, since $\Rr_v = 0$, Condition~\eqref{ort-derive} implies that $v \in [ \A, \A ]^\perp$. Consequently, $v \in [ \A, \A ] \cap [ \A, \A ]^\perp = \{0\}$, so $v = 0$. This imposes $[ \A, \A ] = \{0\}$, contradicting the fact that $(\A, \bullet)$ is non-commutative. Therefore, $[ \A, \A ] \cap [ \A, \A ]^\perp \neq \{0\}$.
\end{proof}

\begin{co}\label{co-Euclidean}
Let $(\A, \bullet, \prs)$ be a left-invariant pseudo-Euclidean nearly associative algebra. If $\prs$ is Euclidean, then $(\A, \bullet, \prs)$ is a quadratic commutative associative algebra.
\end{co}
\begin{proof}
   Since $\prs$ is Euclidean, no non-zero subspace of $\A$ can be degenerate.  Then, according to Proposition~\ref{Non-commutative}, the nearly associative algebra $(\A, \bullet)$ must be commutative. 
   
   In the commutative case, the nearly associative condition reduces to full associativity. Hence, $(\A, \bullet)$ is a commutative associative algebra. Moreover, since $\prs$ is left-invariant, it follows that $(\A, \bullet, \prs)$ is a quadratic commutative associative algebra.  
\end{proof}

\begin{remark}
In \cite{BBEL}, the authors studied left alternative algebras $(\A,\bullet)$ endowed with a positive-definite symmetric bilinear form $\prs$ satisfying the left-invariance condition:
\[
\langle u\bullet v, w\rangle = \langle v, u\bullet w\rangle,
\qquad \text{for all } u,v,w \in \A.
\]
They provided a classification and structural characterization of such algebras.

Now, observe that every  positive-definite (or Euclidean) quadratic commutative associative algebra   naturally satisfies the left-invariance condition, hence it belongs to the class of left alternative algebras considered in \cite{BBEL}.
Consequently, by \cite[Theorem~4.3] {BBEL}, the classification of Euclidean commutative associative algebras $(\A, \bullet, \prs)$ is given as follows: there exists an orthogonal basis $\{e_1,\dots,e_n,f_1,\dots,f_p\}$ of $\A$ such that the only non-vanishing products among the basis elements are
\[
e_i\bullet e_i = e_i, \quad \text{for } i=1,\dots,n,
\]
with the form $\prs$ satisfying 
\[
\langle e_i,e_i\rangle = \lambda_i > 0, \qquad \langle f_j,f_j\rangle = 1,
\]
for all $i=1,\dots,n$ and $j=1,\dots,p$.
\end{remark}

\begin{pr}\label{tow-sided}
Let $(\A, \bullet, \prs)$ be a left-invariant pseudo-Euclidean nearly associative algebra. Then, $[ \A, \A ] \cap [ \A, \A ]^\perp$ is a two-sided ideal of $(\A, \bullet)$, and $[ \A, \A ] + [ \A, \A ]^\perp$ is also a two-sided ideal of $(\A, \bullet)$. 
\end{pr}
\begin{proof}
According to Lemma~\ref{Le2}, we have $[u,v]\bullet w=0$ for all $u,v,w\in \A$, which implies that $\Ll_a = 0$ for all $a \in [\A, \A] \cap [\A, \A]^\perp$. This immediately shows that $[\A, \A] \cap [\A, \A]^\perp$ is a right ideal of $(\A, \bullet)$.

Next, let $u \in \A$, $a \in [\A, \A] \cap [\A, \A]^\perp$, and consider $v + w \in [\A, \A] + [\A, \A]^\perp$ with $v \in [\A, \A]$ and $w \in [\A, \A]^\perp$. Then
$$
\langle u \bullet a, v + w \rangle = \langle u \bullet a, v \rangle + \langle u\bullet a,  w \rangle=\langle u \bullet a,v\rangle,
$$
since $\langle u \bullet a, w \rangle = \langle [u,a], w \rangle = 0$. Moreover, using Equation~\eqref{ort-derive} yields:
\[
\langle u \bullet a, v \rangle = \langle u, v \bullet a \rangle = 0,
\]
because $v \in [\A, \A]$ implies $\Ll_v = 0$. Thus, $u \bullet a \in ([\A, \A] + [\A, \A]^\perp)^\perp = [\A, \A]^\perp \cap [\A, \A]$, which proves that $[\A, \A] \cap [\A, \A]^\perp$ is also a left ideal of $(\A, \bullet)$, and hence a two-sided ideal of $(\A, \bullet)$.

Now, let $v \in [\A, \A] + [\A, \A]^\perp$, $u \in \A$, and $a \in [\A, \A] \cap [\A, \A]^\perp$. Since $[\A, \A] \cap [\A, \A]^\perp$ is a two-sided ideal, we obtain:
\[
\langle u \bullet v, a \rangle \overset{\eqref{ort-derive}}{=} \langle v, u \bullet a \rangle = 0,
\]
and
\[
\langle v \bullet u, a \rangle = \langle u, v \bullet a \rangle = \langle u \bullet a, v \rangle = 0.
\]
This implies that $u \bullet v, v \bullet u \in ([\A, \A] \cap [\A, \A]^\perp)^\perp = [\A, \A]^\perp + [\A, \A]$, which shows that the sum $[\A, \A] + [\A, \A]^\perp$ is a two-sided ideal of $(\A, \bullet)$.\end{proof}

 In \cite{ga}, the authors investigated cyclic pseudo-Euclidean Lie algebras, focusing in particular on the Euclidean case.   Similarly, in \cite{Abid}, cyclic Euclidean Lie algebras were studied in detail.  
In this framework, we show that any left-invariant pseudo-Euclidean nearly associative algebra $(\A, \bullet, \prs)$ naturally induces a cyclic pseudo-Euclidean Lie algebra structure on $(\A^-, \prs)$.

\begin{pr}
Let $(\A, \bullet, \prs)$ be a  left-invariant pseudo-Euclidean nearly associative algebra. Then the pseudo-Euclidean Lie algebra $(\A^-, \prs)$ is cyclic. Moreover, the Levi-Civita product $\star$ associated to $(\A^-, \prs)$ is given by
\[
\langle u \star v, w \rangle = -\langle u, [v, w] \rangle, \quad \text{for all } u,v,w \in \A.
\]
\end{pr}

\begin{proof}
Let $u,v,w \in \A$. By definition of the  product $[u,v] = u \bullet v - v \bullet u$, we have
\[
\langle [u,v], w \rangle = \langle u \bullet v, w \rangle - \langle v \bullet u, w \rangle.
\]
Since $\prs$ is left-invariant, we can deduce that 
\begin{equation*}
\langle  [u,v] , w \rangle + \langle  [v,w] , u \rangle + \langle  [w,u] , v \rangle = 0,
\label{cyclic}
\end{equation*}
for all $u,v,w \in \A$,  meaning that,
$\prs$ is cyclic with respect to the bracket $\br$. 
We have shown that  the pseudo-Euclidean Lie algebra $(\A^-, \prs)$ is cyclic.

Now, according to Proposition~\ref{LVCT}, the Levi-Civita product $\star$ associated to $(\A^-, \prs)$ is given by
\[
2 \langle u \star v, w \rangle = \langle [u,v], w \rangle - \langle [v,w], u \rangle + \langle [w,u], v \rangle.
\]
Since $\prs$ is left-invariant, and using its cyclicity with respect to $\br$ we get
\begin{align*}
2 \langle u \star v, w \rangle &= 2\langle [u,v], w \rangle + 2\langle [w,u], v \rangle \\
&= -2 \langle [v,w], u \rangle.
\end{align*}
Therefore,
\[
\langle u \star v, w \rangle = -\langle u, [v,w] \rangle,
\]
proving the claimed formula and completing the proof.
\end{proof}

\section{Double extension of left-invariant pseudo-Euclidean nearly associative algebras}\label{s6}

In this section, we introduce the notion of  double extension  of left-invariant pseudo-Euclidean nearly associative algebras by commutative associative algebras.  
We then examine particular cases of this construction, including double extensions by a two-dimensional nilpotent commutative associative algebra and double extensions by a one-dimensional trivial algebra.

Let $(\A, \bullet_\A, \prs_\A)$ be a left-invariant pseudo-Euclidean nearly associative algebra, and let $(\h,\star)$ be a commutative associative algebra which is nilpotent of index at most $3$ (so every triple product in $\h$ vanishes).  

To introduce the process of the double extension of $(\A, \bullet_\A, \prs_\A)$ by $(\h, \star)$, consider linear maps $
F, G : \h \to \En(\A)
$,  bilinear maps
$
\om, L : \h \times \h \to \A$ and  $\theta : \h \times \h \to \h^*$, where $\om$ is symmetric. Assume moreover that $G(X)$ is symmetric with respect to $\prs_\A$ for all $X \in \h$. 

To complete the construction, we define the following auxiliary linear and bilinear maps:
\[
\begin{split}
&   S : \A \times \h \longrightarrow \h^*, \quad
R : \h \times \A \longrightarrow \h^*, \quad  N : \h \longrightarrow \En(\h^*),  \\ 
& \mu : \A \times \A \longrightarrow \h^*, \quad  \widetilde{F}, \widetilde{G} : \h \longrightarrow \En(\A \oplus \h^*), \quad 
\widetilde{L} : \h \times \h \longrightarrow\A \oplus \h^*, 
\end{split}
\]
which are given, for all $u, v \in \A$, $X, Y \in \h$, and $f \in \h^*$, by
\begin{equation*}
\begin{split}
 & S(u, X)(Y) = \langle \omega(X, Y), u \rangle_\A,
\quad R(X, u)(Y) = \langle L(X, Y), u \rangle_\A,   \quad 
N(X)(f) = f\circ \Ll_X^\star, \\[3pt]
&\mu(u, v)(X) = \langle F(X)(u), v \rangle_\A, \qquad \widetilde{F}(X)(u + f) = F(X)(u) + S(u, X), \\[3pt]
&\widetilde{G}(X)(u + f) = G(X)(u) + R(X, u)+N(X)(f), \qquad
\widetilde{L}(X, Y) = L(X, Y) + \theta(X, Y),
\end{split}
\label{eq:def-maps}
\end{equation*}
where $\Ll^\star$ denotes the left multiplication operator in $(\h,\star)$.

\begin{Le} \label{pr-central1}
Let $\widetilde{\A} := \A \oplus \h^*$ be the vector space equipped with the product
\[
(u + f) \widetilde{\bullet} (v + g) := u \bullet_\A v + \mu(u, v), 
\quad \text{for all }  u, v \in \A, \; f, g \in \h^*.
\]
Then $(\widetilde{\A}, \widetilde{\bullet})$ is a nearly associative algebra if and only if the following identity holds:
\begin{equation}
F(X) \circ \Ll_u = \Rr_u^* \circ F(X), 
\quad \text{ for all  } u \in \A, \; X \in \h.
\label{1-cycle}
\end{equation}
\end{Le}

\begin{proof}
The pair $(\widetilde{\A}, \widetilde{\bullet})$ is a nearly associative algebra if and only if the bilinear map $\mu$ satisfies Condition~\eqref{ex-central}.  
It is straightforward to show that Condition~\eqref{ex-central} is equivalent to the Identity~\eqref{1-cycle}.
\end{proof}

\begin{Le}\label{almost}
Let $(\h, \A, F, G, L)$ be a context of an almost semi-direct product of nearly associative algebras. Then
   $(\h, \widetilde{\A}=\A\oplus \h^*,  \widetilde{F}, \widetilde{G}, \widetilde{L})$ defines a context of an almost semi-direct product of nearly associative algebras if and only if the following identities hold for all $X,Y,Z, T\in \h$
\begin{equation}\label{les equation}
    \begin{cases}
       \Rr^\A_{\omega(X,Y)}=F(Y)^*\circ G(X),\\ \Rr^\A_{L(X,Y)}+F(X\star Y)=F(X)^*\circ F(Y)^*,\\  F(X)^*\circ F(Y)=F(Y)\circ G(X),\\ 
       F(X)^* (\omega(Y, Z))=G(Y)\big(L(X,Z)\big)+L(Y,X\star Z),\\ 
       G(X) \big(\omega(Y,Z)\big)=F(Z)^*\big(L(X,Y)\big)+\omega(X\star Y,Z),\\ 
       F(X)^*\big(L(Y,Z)\big)+\omega(X,Y\star Z)
    =F(Z) \big(L(X,Y)\big) +L(X\star Y,Z) ,\\ 
    \theta(X\star Y,Z)(T)+\langle \omega(Z,T),L(X,Y)\rangle_\A
    = \theta(Y,Z\star X)(T)+\langle L(Y,T),L(Z,X)\rangle_\A+\theta(Z,X)(Y\star T), 
    \end{cases}
\end{equation}
where $\Rr^\A$ denotes the right multiplication operator in $(\A, \bullet_\A)$.
\end{Le}

\begin{proof}
Let $X, Y, Z, T\in \h$, $u,v \in \A$, and $f,g \in \h^*$.  
Straightforward computations show that:
\begin{enumerate}

\item The condition $\widetilde{F}(X)\big((u+f)\widetilde{\bullet} (v+g)\big) = (v+g)\widetilde{\bullet} \widetilde{G}(X)(u+f)$ is equivalent to
    \[
    F(X)(u\bullet_{\mathcal{A}} v)=v\bullet_{\mathcal{A}} G(X)(u) \quad \text{and} \quad \Rr^{\mathcal{A}}_{\omega(X,Y)}=F(Y)^*\circ G(X).
    \]
    
\item The equation $\widetilde{F}(X)(u+f)\widetilde{\bullet}(v+g) = \widetilde{G}(X)\big((v+g)\widetilde{\bullet} (u+f)\big)$ is equivalent to
    \[
    F(X)(u)\bullet_{\mathcal{A}} v = G(X)(v\bullet_{\mathcal{A}} u) \quad \text{and} \quad \Rr^{\mathcal{A}}_{L(X,Y)}+F(X\star Y)=F(X)^*\circ F(Y)^*.
    \]

 \item The relation $(u+f)\widetilde{\bullet} \big(\widetilde{F}(X)(v+g)\big) = \big(\widetilde{G}(X)(u+f)\big)\widetilde{\bullet} (v+g)$ is equivalent to
    \[
    u\bullet_{\mathcal{A}} F(X)(v)= G(X)(u)\bullet_{\mathcal{A}} v \quad \text{and} \quad F(X)^*\circ F(Y)=F(Y)\circ G(X).
    \]

    \item The condition $\widetilde{F}(Y)\circ \widetilde{F}(X) = \widetilde{G}(X)\circ \widetilde{G}(Y)$ is equivalent to
    \[
    F(Y)\circ F(X)= G(X)\circ G(Y) \quad \text{and} \quad F(X)^*(\omega(Y, Z))=G(Y) (L(X,Z))+L(Y,X\star Z).
    \]
    
    \item The relation $\widetilde{F}(Y)\circ \widetilde{G}(X) = \Rr_{\widetilde{L}(X,Y)}+\widetilde{F}(X\star Y)$ is equivalent to
    \[
    F(Y)\circ G(X)= \Rr^{\mathcal{A}}_{L(X,Y)}+F(X\star Y) \quad \text{and} \quad G(X)\big(\omega(Y,Z)\big)=F(Z)^*\big(L(X,Y)\big)+\omega(X\star Y,Z).
    \]

\item The equation $\widetilde{G}(Y)\circ \widetilde{F}(X) = \Ll_{\widetilde{L}(X,Y)}+\widetilde{G}(X\star Y)$ is equivalent to
    \[
    G(Y)\circ F(X)= \Ll^{\mathcal{A}}_{L(X,Y)}+G(X\star Y) \quad \text{and} \quad F(X)^*\big(L(Y,Z)\big)+\omega(X,Y\star Z) = F(Z)\big(L(X,Y)\big)    + L(X\star Y,Z).
    \]

    \item Finally, the relation $\widetilde{L}(X\star Y, Z)+\widetilde{F}(Z) \big(\widetilde{L}(X, Y) \big)= \widetilde{L}(Y, Z\star X)+\widetilde{G}(Y) \big(\widetilde{L}(Z, X)\big)$ is equivalent to
    \[
    L(X\star Y, Z)+F(Z) \big(L(X,Y)\big)
    = L(Y, Z\star X)+G(Y)\big( L(Z,X)\big) ,
    \]
    and
    \[
    \theta(X\star Y,Z)(T)+\langle \omega(Z,T),L(X,Y)\rangle_\A
    = \theta(Y,Z\star X)(T)+\langle L(Y,T),L(Z,X)\rangle_\A+\theta(Z,X)(Y\star T).
    \]
\end{enumerate}

Since $(\h, \A, F, G, L)$ is a context of an almost semi-direct product of nearly associative algebras, it follows that $(\h, \widetilde{\A}=\A\oplus \h^*,  \widetilde{F}, \widetilde{G}, \widetilde{L})$ is a context of an almost semi-direct product of nearly associative algebras if and only if the  Conditions~\eqref{les equation} are satisfied. 
\end{proof}

Now, we are in a position to introduce the notion of   double extension  of a left-invariant pseudo-Euclidean nearly associative algebra by a commutative associative algebra which is nilpotent of index at most $3$.

\begin{theo}\label{G-double}
Let $(\A, \bullet_\A, \prs_\A)$ be a left-invariant pseudo-Euclidean nearly associative algebra, and let $(\h,\star)$ be a commutative associative algebra which is nilpotent of index at most $3$. 
Let $(\h, \A, F, G, L)$ be a context of an almost semi-direct product of nearly associative algebras such that $G(X)$ is symmetric with respect to $\prs_\A$ for any $X\in \h$. Let 
\(
\omega:\h\times\h\to\mathcal A 
\)  be a symmetric  bilinear map, and let \(
\theta:\mathcal \h\times\mathcal \h\to\mathcal \h^*
\)
be a bilinear map.
For $X\in\h$ and $u,v\in\mathcal A$, define the linear functionals $\langle L(X,\cdot),u\rangle_{\mathcal A}, \, \langle F(\cdot)(u),v\rangle_{\mathcal A}, \, \langle \omega(X,\cdot),u\rangle_{\mathcal A} \in \h^* $ by 
\begin{align*}
&\big(\langle L(X,\cdot),u\rangle_{\mathcal A}\big)(Y)=\langle L(X,Y),u\rangle_{\mathcal A},
\qquad \big(\langle F(\cdot)(u),v\rangle_{\mathcal A}\big)(Y)=\langle F(Y)(u),v\rangle_{\mathcal A},\\
&\big(\langle \omega(X,\cdot),u\rangle_{\mathcal A}\big)(Y)=\langle\omega(X,Y),u\rangle_{\mathcal A},
\end{align*}
for all $Y\in\h$. Define a bilinear product $\bullet$ on $\overline{\mathcal{A}}:=\h\oplus\mathcal A\oplus \h^*$ by
\begin{equation}\label{eq:prod-tildeA}
\begin{aligned}
X \bullet Y &= X \star Y + L(X,Y) + \theta(X,Y), & X \bullet u &= G(X)(u) + \langle L(X,\cdot), u \rangle_{\mathcal{A}}, \\
u \bullet X &= F(X)(u) + \langle \omega(X,\cdot), u \rangle_{\mathcal{A}}, & u \bullet v &= u \bullet_{\mathcal{A}} v + \langle F(\cdot)(u), v \rangle_{\mathcal{A}}, \\
X \bullet f &= f \circ \Ll_X^\star, & f \bullet X &= u \bullet f = f \bullet u = f \bullet g = 0,
\end{aligned}
\end{equation}
for all $X,Y\in\h$, $u,v\in\mathcal A$, and $f,  g  \in\h^*$, where $\Ll_X^\star$ denotes left multiplication by $X$ in $(\h,\star)$.

Then $(\overline{\mathcal{A}},\bullet)$ is a nearly associative algebra if and only if the Conditions \eqref{1-cycle} and \eqref{les equation} are satisfied.

Moreover, the bilinear form $
\prs : \overline{\mathcal{A}} \times \overline{\mathcal{A}} \longrightarrow \K$ defined by
$$
\langle u+X+f,\, v+Y+g\rangle:= \langle u, v\rangle_\A + f(Y) + g(X),
$$
is a left-invariant pseudo-Euclidean scalar product on $(\overline{\mathcal{A}}, \bullet)$ if and only if $\theta$ satisfies
\begin{equation}
    \theta(X, Y)(Z)=\theta(X, Z)(Y), \quad \text{for all }  X, Y, Z\in \h.\label{cyclic} 
    \end{equation}
    
In these conditions, $(\overline{\mathcal{A}}, \bullet, \prs)$ is a left-invariant pseudo-Euclidean nearly associative algebra.
\end{theo}

\begin{proof}
 Suppose first that Conditions \eqref{1-cycle} and \eqref{les equation} are satisfied.  According to Lemma~\ref{pr-central1}, the Condition \eqref{1-cycle} ensures that the product on 
\(\widetilde{\A}  = \A \oplus \h^*\) given by
\[
(u+ f)\widetilde{\bullet}(v+g) = u\bullet_\A v + \langle F(\cdot)(u), v\rangle_\A,
\]
for all $u,v\in \A$ and $f,g \in \h^*$, defines a nearly associative algebra. In other words, $( \widetilde{\A},\widetilde{\bullet})$ is an annihilator  extension of \((\A,\bullet_\A)\) by \(\h^*\) means of the bilinear map \(\varphi\) defined by \(\varphi(u,v)(X) := \langle F(X)(u), v\rangle_\A\) for all \(u,v\in \A\) and \(X\in \h\).  

Since \((\h, \A, F, G, L)\) forms a context of  an almost semi-direct product of nearly associative algebras with \(G(X)\) symmetric with respect to \(\prs_\A\) for all \(X\in \h\), and the system \eqref{les equation} is satisfied, it follows from Lemma~\ref{almost} that 
\((\h, \widetilde{\A}, \widetilde{F}, \widetilde{G}, \widetilde{L})\) is itself a context of an almost semi-direct product of nearly associative algebras. Consequently, \((\overline{\A} = \h \oplus \A \oplus \h^*, \bullet)\) is a nearly associative algebra.

Suppose now that \((\overline{\A} = \h \oplus \A \oplus \h^*, \bullet)\) is a nearly associative algebra. By Conditions \eqref{eq:prod-tildeA}, the subspace \(\widetilde{\A}  = \A \oplus \h^*\) is closed under $\bullet$, and its induced product is given by
\[
(u+ f)\widetilde{\bullet}(v+g) = u\bullet_\A v + \langle F(\cdot)(u), v\rangle_\A,
\]
for all $u,v\in \A$ and $f,g \in \h^*$. Hence $(\widetilde{\A}, \widetilde{\bullet})$ is a nearly associative algebra. It then follows from Lemma~\ref{pr-central1} that Condition \eqref{1-cycle} holds.

Furthermore, Conditions \eqref{eq:prod-tildeA} can be rewritten in terms of the product $\widetilde{\bullet}$ on $\widetilde{\A}$ as
\begin{align*}
    &(u+f) \bullet (v+g) = (u+f) \widetilde{\bullet} (v+g), \quad
(u+f) \bullet X = \widetilde{F}(X)(u+f),\\
& X \bullet (u+f) = \widetilde{G}(X)(u+f), \quad
X \bullet Y = X \star Y + \widetilde{L}(X,Y),
\end{align*}
for all $u,v \in \A,$ $ f,g\in \h^*,$ and $ X,Y \in \h.$ Therefore, $(\h, \widetilde{\A}=\A\oplus \h^*,  \widetilde{F}, \widetilde{G}, \widetilde{L})$ is a context of an almost semi-direct product of nearly associative algebras. Since \((\h, \A, F, G, L)\) forms a context of an almost semi-direct product of nearly associative algebras and \(G(X)\) symmetric with respect to \(\prs_\A\) for all \(X\in \h\), Lemma~\ref{almost} implies that the system \eqref{les equation} is satisfied.

Finally, one can easily verify that the bilinear form \(\prs\) defines a left-invariant pseudo-Euclidean scalar product on \((\overline{\A}, \bullet)\) if and only if Condition~\eqref{cyclic} holds. Hence $(\overline{\mathcal{A}}, \bullet, \prs)$ is a left-invariant pseudo-Euclidean nearly associative algebra.
\end{proof}

\begin{Def}
The left-invariant pseudo-Euclidean nearly associative algebra $(\overline{\mathcal{A}}, \bullet, \prs)$ constructed in Theorem \ref{G-double} is called the \emph{generalized double extension}  of the left-invariant pseudo-Euclidean nearly associative algebra $(\mathcal{A}, \bullet_\A, \prs_\A)$ by the nilpotent  commutative associative algebra  $(\h, \star)$ of nilpotency index at most $3$ by means of $(F, G, L, \theta, \omega)$.
\end{Def}

\subsection{
Some particular cases}

\subsubsection{The one-dimensional case}

We  study the particular case where the algebra $\h$ is $1$-dimensional. Let $\h = \mathbb{K} d$ and $\h^* = \mathbb{K} e$, where $e(d) = 1$.  
In this setting,  the maps $F$, $G$, $L$, $\omega$, $\theta$, and the product  $\star$ in $\h$ are given by:
\[
F(d) = \delta, \qquad 
G(d) = D, \qquad 
L(d,d) = a_0, \qquad 
\omega(d,d) = b_0, \qquad 
\theta(d,d) = \lambda e, \qquad 
d \star d = 0,
\]
where $D \in \mathrm{End}(\A)$ is symmetric with respect to $\prs_\A$, $\de \in \mathrm{End}(\A)$, $a_0, b_0 \in \A$, and $\lambda \in \mathbb{K}$. Hence, for all $u, v \in \A$, we have:
\begin{align*}
\langle L(d, \cdot), u \rangle_\A  &=  \langle a_0, u \rangle_\A e,\quad
\langle F(\cdot)(u), v \rangle_\A = \langle \delta(u), v \rangle_\A e,\quad
\langle \omega(d, \cdot), u \rangle_\A  = \langle b_0, u \rangle_\A e.
\end{align*}
Consequently, the fact that $(\h, \A, F, G, L)$ defines an almost semi-direct product of nearly associative algebras and that the maps $F, G, L, \omega, \theta$ satisfy Conditions~\eqref{1-cycle}, \eqref{les equation} and \eqref{cyclic} is equivalent to the following system:
\begin{equation}\label{sys-dim1}
\begin{cases}
\delta(u \bullet_\A v) = v \bullet_\A D(u),\quad 
\delta(u) \bullet_\A v = D(v \bullet_\A u),\quad 
u \bullet_\A \delta(v) = D(u) \bullet_\A v,\\[4pt]
\delta \circ D = \delta^* \circ \delta = \Rr_{a_0} = \delta^2 = D^2, \quad 
\delta \circ \Ll_u = \Rr_u^* \circ \delta,\\[4pt]
D \circ \delta = \Ll_{a_0} = \Rr_{b_0}, \quad 
\delta(a_0) = \delta^*(a_0) = \delta^*(b_0) = D(a_0) = D(b_0), \\[4pt]
\langle a_0, a_0 - b_0 \rangle_\A = 0, 
\end{cases}
\end{equation}
for all $u, v \in \A$.

In this case, the product of the nearly associative algebra 
\(
\overline{\A} = \mathbb{K} d \oplus \A \oplus \mathbb{K} e
\)
is given by
\begin{equation}\label{Produit-dim1}
    \begin{aligned}
d \bullet d &= a_0 + \lambda e, & d \bullet u &= \delta(u) + \langle a_0, u \rangle_\A e, \\ 
u \bullet d &= D(u) + \langle b_0,  u \rangle_\A e, &  u \bullet v &= u \bullet_\A v + \langle \delta(u), v \rangle_\A e,\\
 \overline{\A} \bullet e &  = e \bullet \overline{\A}= \{0\},
    \end{aligned}
\end{equation}
for all $u, v \in \A$.
Moreover, the symmetric bilinear form $\prs$ on $\overline{\A}$ is defined by
\begin{equation}\label{metric-dim1}
\langle u, v \rangle = \langle u, v \rangle_\A,  \quad 
\langle d, \A \rangle = \langle e, \A \rangle = \{0\}, \quad 
\langle d, e \rangle  = 1, \quad \langle e, e \rangle = \langle d, d \rangle = 0,
\end{equation}
for all $u, v \in \A$.

The following theorem summarizes this double extension of left-invariant pseudo-Euclidean nearly associative algebra by one-dimensional algebras.
\begin{theo}
Let $(\mathcal{A}, \bullet_\A, \prs_\A)$ be a left-invariant pseudo-Euclidean nearly associative algebra.  
Let $\h  = \mathbb{K}d$ be a one-dimensional algebra and $\h^* = \mathbb{K}e$ denotes its dual. Assume that there exist  linear maps $\de, D \in \mathrm{End}(\A)$, elements $a_0, b_0 \in \A$, and a scalar $\lambda \in \K$ such that $D$ is symmetric with respect to  $\prs_\A$, and the system~\eqref{sys-dim1} is satisfied.  Then, the vector space 
\(
\overline{\A} = \h \oplus \A \oplus \h^*
\)
equipped with the product $\bullet$ defined by~\eqref{Produit-dim1} and the symmetric bilinear form $\prs$ given by~\eqref{metric-dim1} 
is a left-invariant pseudo-Euclidean nearly associative algebra.

We call $(\overline{\A}, \bullet, \prs)$ the  double extension  of $(\mathcal{A}, \bullet_\A, \prs_\A)$ by the one-dimensional trivial algebra $\h$ by means of $(\delta, D, a_0, b_0, \la)$.
\end{theo}

\subsubsection{The case where $\h$ is a two-dimensional nilpotent commutative associative algebra}

In this subsection, we make use of $2$-dimensional nilpotent commutative associative algebras. To this end, we briefly recall the classification of such algebras established by Graaf in \cite{Graf}, which will play a fundamental role in the construction of our double extensions.

\begin{pr} [see \cite{Graf}]
Any $2$-dimensional nilpotent commutative associative algebra is isomorphic to the commutative associative algebra $(\h  = \mathbb{K}d_1\oplus \mathbb{K} d_2,  \star)$ whose product is defined by 
\[
d_1 \star d_1 = s d_2, \quad \text{with } s\in \{0, 1\}.
\]
\end{pr}

Now, we assume that $\dim \h = 2$. We denote $\h := \mathbb{K}d_1\oplus \mathbb{K} d_2$ and $\h^* := \mathbb{K}e_1\oplus \K e_2$, where $e_i(d_i) = 1$ and $e_i(d_j) = 0$ for $i \neq j$, with $i,j \in \{1,2\}$. 
In this setting, the maps $F$, $G$, $L$, $\omega$, $\theta$ and the product  $\star$ in $\h$ are defined by
  \begin{align*}
           & F(d_i) = \delta_i, \quad G(d_i)=D_i,  
\quad 
L(d_i,d_j) = a_{ij},\quad \omega(d_i,d_j)=b_{ij},\quad 
\quad 
\theta(d_i,d_j) = \al_{ij} e_1+\beta_{ij} e_2, \\ 
           &  d_1\star d_1=s d_2,   \end{align*}
where $i,j \in \{1,2\}$, $D_i \in \mathrm{End}(\mathcal{A})$ are symmetric with respect to $\prs_{\mathcal{A}}$, $\de_i \in \mathrm{End}(\mathcal{A})$, $a_{ij}, b_{ij} \in \mathcal{A}$, $\al_{ij}, \beta_{ij}\in \mathbb{K}$ and $s\in\{0,1\}$.  Hence, for all $ u,v \in \mathcal{A},$ and $i,j\in\{1, 2\},$ we have:
\begin{align*}
&\langle L(d_i, .), u\rangle_\A=\langle a_{i1}, u\rangle e_1+\langle a_{i2}, u\rangle e_2,\;\; \langle F(.)(u), v\rangle_\A=\langle \de_1(u), v\rangle e_1+ \langle \de_2(u), v\rangle e_2,\\
&\langle \omega(d_i, .) , u\rangle_\A=\langle b_{i1}, u\rangle e_1+\langle b_{i2}, u\rangle e_2 ,
\end{align*}
Consequently, the fact that $(\h, \A, F, G, L)$ defines an almost semi-direct product of nearly associative algebras and that $F, G, L, \omega, \theta$ satisfy Conditions~\eqref{1-cycle}, ~\eqref{les equation} and \eqref{cyclic} is equivalent to the following system:

\begin{equation}\label{sys-dim2}
\left\{
\begin{aligned}
&\delta_i(u \bullet_\A v) = v \bullet_\A D_i(u), \; 
  \delta_i(u) \bullet_\A v = D_i(v \bullet_\A u), \; 
  u \bullet_\A \delta_i(v) = D_i(u) \bullet_\A v, \\
&\delta_j \circ \delta_i = D_i \circ D_j,\;\delta_i \circ \Ll_u = \Rr_u^* \circ \delta_i,\; \de_i^*\circ \de_j=\de_j\circ D_i,\; \Rr_{b_{ij}}=\de_j^*\circ D_i,
\\
& \delta_1 \circ D_1 = \Rr_{a_{11}} + s \delta_2, \; D_1 \circ \delta_1 = \Ll_{a_{11}} + s D_2,\; \Rr_{a_{11}}+s\de_2=(\de_1^*)^2,\\
& \delta_i \circ D_j=\de_i^*\circ\de_j^* = \Rr_{a_{ij}}, \;
  D_j \circ \delta_i = \Ll_{a_{ij}},\; \de_i^*(b_{kj})=D_k(a_{ij}),\; D_i(b_{jk})=\de^*_k(a_{ij}),  \, (i,j) \neq (1,1), 
   \\ 
   & \delta_1^*(b_{k1}) = D_k(a_{11}) + s a_{k2},\; D_1(b_{1k}) = \de^*_k(a_{11}) + s b_{2k}, \\\
   &
 s a_{21} + \delta_1(a_{11}) = sa_{12}+ D_1(a_{11}),\;\de_1(a_{21})=D_1(a_{12}), \; sa_{22}+\de_2(a_{11})=D_1(a_{21}), \\
 &\de_2^*(a_{12})=\de_2(a_{21})=D_1(a_{22}),\;
  \de_1^*(a_{21})=\de_1(a_{12}) =s a_{22} + D_2(a_{11}), \\
  &\de_2^*(a_{21})=\de_1(a_{22})=D_2(a_{12}), \; \de_1^*(a_{22})=\de_2(a_{12})=D_2(a_{21}), \; \de_2(a_{22})=D_2(a_{22}),\\
  & \de_1^*(a_{11})+sb_{12}=\de_1(a_{11})+sa_{21},\; \de_1^*(a_{12})=\de_2(a_{11})+sa_{22},\; \de_2^*(a_{11})+sb_{22}=\de_1(a_{21}), 
\\
 & s\al_{21} + \langle b_{11}, a_{11} \rangle_{\A} =2s\al_{12}+\langle a_{11},a_{11}\rangle_{\A},\; s\be_{21} + \langle b_{12}, a_{11}\rangle_{\A}=s \be_{12}+\langle a_{12}, a_{11}\rangle_{\A}= \langle b_{11}, a_{21}\rangle_{\A}, \\
 & \langle b_{21}, a_{11}\rangle_{\A}=\langle a_{11}, a_{21} \rangle_{\A} =\langle b_{11}, a_{12}\rangle_{\A}- s \al_{22},\;  s\be_{22}+ \langle b_{22}, a_{11}\rangle_{\A} = \langle a_{12}, a_{21}\rangle_{\A}=\langle b_{11}, a_{22}\rangle_{\A},\\
 &
 \langle b_{22}, a_{21}\rangle_{\A}=\langle a_{12}, a_{22}\rangle_{\A}, \;\langle b_{12}, a_{12} \rangle_{\A} = s \be_{22}+ \langle a_{22}, a_{11}\rangle_{\A}= \langle b_{12}, a_{21}\rangle_{\A}, \\ & \langle b_{21}, a_{12}\rangle_{\A}= \langle a_{21}, a_{21}\rangle_{\A}, \; \langle b_{22}, a_{12}\rangle_{\A}= \langle a_{22}, a_{21}\rangle_{\A},\;  \langle b_{21}, a_{22}\rangle_{\A}=\langle a_{12}, a_{22}\rangle_{\A},\\ & \langle b_{2}, a_{21}\rangle_{\A}= \langle a_{12}, a_{12}\rangle_{\A}, \; \langle a_{12}, a_{22}\rangle_{\A}= \langle b_{12}, a_{22}\rangle_{\A}, \; \langle b_{22}, a_{22}\rangle_{\A}=\langle a_{22}, a_{22}\rangle_{\A},\\
 &\alpha_{12} = \beta_{11}, \; \alpha_{22} = \beta_{21},\quad s\in\{0,1\}  
 \end{aligned}
\right.
\end{equation}
for all $u,v\in \A$ and $i,j,k\in \{1,2\}$.

In this case, the product of the nearly associative algebra 
\(
\overline{\A} = \mathbb{K} d_1 \oplus \K d_2 \oplus \A \oplus \mathbb{K} e_1 \oplus \K e_2
\)
is given by
\begin{equation}\label{Produit-dim2}
\begin{split}
&d_1 \bullet d_1 = s d_2 +a_{11}+ \al_{11}e_1+\beta_{11}e_2,\quad
d_i \bullet d_j = a_{ij}+ \al_{ij}e_1+\beta_{ij}e_2,\; (i,j)\neq(1,1)\\[3pt]
&d_i \bullet u =  D_i(u)   + \langle a_{i1}, u \rangle_\A e_1+\langle a_{i2}, u \rangle_\A e_2,\quad
u \bullet d_i =  \de_i(u)   + \langle b_{i1}, u \rangle_\A e_1+\langle b_{i2}, u \rangle_\A e_2,\\[3pt]
& u \bullet v = u \bullet_\A v + \langle \delta_1(u), v \rangle_\A e_1 + \langle \delta_2(u), v \rangle_\A e_2,
  \quad d_1\bullet e_2= e_2 \circ \Ll_{d_1}^{\star}, \\
&  d_1 \bullet e_1= d_2 \bullet e_i= e_i \bullet d_j =u \bullet e_i = e_i \bullet u = e_i \bullet e_j =0,
\end{split}
\end{equation}
for all $u, v \in \A$ and $i,j\in\{1,2\}$.
Moreover, the symmetric bilinear form $\prs$ on $\overline{\A}$ is defined by
\begin{equation}\label{metric-dim2}
\begin{split}
&\langle u, v \rangle = \langle u, v \rangle_\A,  \quad 
\langle d_i, \A \rangle = \langle e_i, \A \rangle =\{0\}, \quad 
\langle d_i, e_i \rangle = 1, \\
&\langle e_i, e_j \rangle = \langle d_i, d_j \rangle =\langle d_1, e_2 \rangle=\langle d_2, e_1 \rangle= 0,
\end{split}
\end{equation}
for all $u, v \in \A$, $i,j\in \{1,2\}$ and $i \neq j$.

The following theorem summarizes this double extension of a left-invariant pseudo-Euclidean nearly associative algebra by two-dimensional nilpotent commutative associative algebras.
\begin{theo}
Let $(\mathcal{A}, \bullet_\A, \prs_\A)$ be a left-invariant pseudo-Euclidean nearly associative algebra.  
Let $(\h = \mathbb{K} d_1 \oplus \mathbb{K} d_2, \star)$ be a two-dimensional nilpotent commutative associative algebra whose product is defined by 
$
d_1 \star d_1 = s\, d_2$, where $s \in \{0,1\},$
and let $\h^* = \mathbb{K} e_1 \oplus \mathbb{K} e_2$ be its dual space such that $e_i(d_i) = 1$, for all $i,j \in \{1,2\}$.  

Assume that there exist linear maps $\delta_i, D_i \in \mathrm{End}(\A)$, elements $a_{ij}, b_{ij} \in \A$, and scalars $\alpha_{ij}, \beta_{ij} \in \mathbb{K}$ such that each $D_i$ is symmetric with respect to $\prs_{\mathcal{A}}$, and the system \eqref{sys-dim2} is satisfied. Then, the vector space 
$
\overline{\A} = \h \oplus \A \oplus \h^*
$
equipped with the product $\bullet$ defined by~\eqref{Produit-dim2} and the symmetric bilinear form $\prs$ given by~\eqref{metric-dim2} 
is a left-invariant pseudo-Euclidean nearly associative algebra.  

We call $(\overline{\A}, \bullet, \prs)$ the  double extension  of $(\mathcal{A}, \bullet_\A, \prs_\A)$ by the two-dimensional nilpotent commutative associative algebra $(\h, \star)$ by means of the data $(\delta_i, D_i, a_{ij}, b_{ij}, \alpha_{ij}, \beta_{ij})$.
\end{theo}

\section{Structures and inductive descriptions of left-invariant pseudo-Euclidean nearly associative algebras}\label{s7}

In this section, we show that every left-invariant pseudo-Euclidean nearly associative algebra can be obtained by a double extension either by a two-dimensional commutative associative algebra or by a one-dimensional algebra.  
Moreover, we prove that any left-invariant pseudo-Euclidean nearly associative algebra can be constructed through a sequence of successive double extensions starting from a quadratic commutative associative algebra.  
The structure  of quadratic commutative associative algebras is fully studied in \cite{BenBar}.

\begin{pr}\label{reduit1}
Let $(\A, \bullet, \prs)$ be a left-invariant pseudo-Euclidean nearly associative algebra.  
Let $I \subseteq [\A, \A] \cap [\A, \A]^\perp$ be a two-sided ideal of $(\A, \bullet)$.   
Denote by $\mathcal{B} := I^\perp / I$ and $\h := \A / I^\perp$, and let  
$\pi_{\mathcal{B}} : I^\perp \to \mathcal{B}$ and $\pi_{\h} : \A \to \h$ be the canonical projections. Then:
\begin{enumerate}
    \item $I^\perp$ is a two-sided ideal of $(\A, \bullet)$. Moreover, we have $I^\bot\bullet I=I\bullet \A=\{0\}.  $
    
    \item $\mathcal{B}$ inherits a canonical structure of a left-invariant pseudo-Euclidean nearly associative algebra, defined by
    \[
      \pi_\mathcal{B}(u) \bullet_{\mathcal{B}} \pi_\mathcal{B}(v) := \pi_\mathcal{B}(u \bullet v), 
      \qquad 
      \langle \pi_\mathcal{B}(u), \pi_\mathcal{B}(v) \rangle_{\mathcal{B}} := \langle u, v \rangle,
    \]
    for all $u, v \in I^\perp$.
    
    \item $\h$ inherits a canonical structure of a commutative associative algebra, defined by
    \[
      \pi_{\h}(u) \bullet_{\h} \pi_{\h}(v) := \pi_{\h}(u \bullet v),
    \]
    for all $u, v \in \A$. Moreover, $\h$ is nilpotent of index at most $3$, i.e.,
    \(
      \pi_{\h}(u) \bullet_{\h} \pi_{\h}(v) \bullet_{\h} \pi_{\h}(w)= 0,
    \)
    for all $u, v, w \in \A$.
\end{enumerate}
\end{pr}

\begin{proof}
\begin{enumerate}
\item Let $u \in I$, $v \in I^\perp$, and $w \in \A$. Then
\[
\langle w \bullet v, u \rangle = \langle v, w \bullet u \rangle = 0,
\]
since $I$ is a two-sided ideal of $(\A, \bullet)$. Hence $I^\perp$ is a right ideal of $(\A, \bullet)$.  
Moreover, as $u \in I \subseteq [\A,\A] \cap [\A,\A]^\perp$, relation \eqref{ort-derive} implies that $\Rr_u = \Rr_u^*$. Thus,
\[
\langle v \bullet w, u \rangle = \langle w, v \bullet u \rangle = \langle w \bullet u, v \rangle = 0,
\]
which shows that $I^\perp$ is also a left ideal of $(\A, \bullet)$ and $I^\bot\bullet I=\{0\}$. Therefore $I^\perp$ is a two-sided ideal of $(\A, \bullet)$. Moreover, since $I \subseteq [\A,\A] \cap [\A,\A]^\perp$, then according to  Lemma \ref{Le2} implies that $\Ll_u=0$, for all $u\in I $ and hence $I\bullet \A=\{0\}$ .

\item Since $I^\perp$ is a two-sided ideal, it follows directly that the structure defined by
\[
\pi_{\mathcal{B}}(u) \bullet_{\mathcal{B}} \pi_{\mathcal{B}}(v) := \pi_{\mathcal{B}}(u \bullet v), 
\qquad 
\langle \pi_{\mathcal{B}}(u), \pi_{\mathcal{B}}(v) \rangle_{\mathcal{B}} := \langle u, v \rangle,
\]
for all $u, v \in I^\perp$, endows $\mathcal{B}$ with the structure of a left-invariant pseudo-Euclidean nearly associative algebra.

\item Clearly, 
\(
\pi_{\h}(u) \bullet_{\h} \pi_{\h}(v) := \pi_{\h}(u \bullet v), 
\)
for all $u, v \in \A$, defines a nearly associative algebra structure on $\h$. Since $[\A, \A] \subseteq I^\perp$, we obtain 
\(
\pi_{\h}(u) \bullet_{\h} \pi_{\h}(v) = \pi_{\h}(v) \bullet_{\h} \pi_{\h}(u),\) for all $u,v\in \A$ and hence $\h$ is a commutative associative algebra.    
Furthermore, for all $u,v,w \in \A$ and $z \in I$, we have
\[
\langle (u \bullet v) \bullet w, z \rangle 
= \langle w, (u \bullet v) \bullet z \rangle 
\overset{\eqref{nearly-associative}}{=} \langle w, v \bullet (z \bullet u) \rangle = 0,
\]
since $z \in I \subseteq [\A, \A] \cap [\A, \A]^\perp$, and by Lemma~\ref{Le2} we have $\Ll_z = 0$.  
Thus $(u \bullet v) \bullet w \in I^\perp$, and consequently 
\[
\pi_{\h}\bigl((u \bullet v) \bullet w\bigr) = 0.
\]
Since $\h$ is an associative algebra, this proves that $\h$ is nilpotent of index at most $3$.
\end{enumerate}
\end{proof}

\begin{theo}\label{Lemme-ex}
Let $(\A, \bullet, \prs)$ be a left-invariant pseudo-Euclidean non-commutative nearly associative algebra.  
Assume that there exists a totally isotropic two-sided ideal 
\(
I \subseteq [\A, \A] \cap [\A, \A]^\perp
\)
of $(\A, \bullet)$. Then $(\A, \bullet, \prs)$ is a double extension of the left-invariant pseudo-Euclidean nearly associative algebra 
\(
(\mathcal{B} := I^\perp / I, \bullet_\mathcal{B}, \prs_\mathcal{B})
\)
by the commutative associative algebra 
\(
\h := \A / I^\bot,
\)
which is nilpotent of index at most $3$, by means of the data $(F, G, L, \omega, \theta)$.
\end{theo}

\begin{proof}
Since $I \subseteq [\A, \A] \cap [\A, \A]^\perp$ is a
two-sided ideal of $(\A, \bullet)$, Proposition~\ref{reduit1} implies that $I^\perp$ is also a two-sided ideal of $(\A, \bullet)$, and $I\bullet \A=I^\bot\bullet I=\{0\}$. Moreover, 
\(
\mathcal{B}  = I^\perp / I
\)
is a left-invariant pseudo-Euclidean nearly associative algebra, and
\(
\h  = \A / I^\bot
\)
is a commutative associative algebra, which is nilpotent of index at most $3$.  

Now, since $I \subseteq [\A, \A] \cap [\A, \A]^\perp$ is totally isotropic, we can choose a vector subspace $B$ such that
\(
I^\perp = B \oplus I.
\)
There exists a totally isotropic vector subspace $V$ such that
\(
B^\perp = I \oplus V.
\)
The restriction $\prs_B := \prs|_{B \times B}$ of $\prs$ to $B$ is non-degenerate, and the map
\(
I \longrightarrow V^*, \quad u \mapsto \langle u, \cdot \rangle_B
\)
is an isomorphism. Consequently, the pseudo-Euclidean vector space $(B^\perp, \prs_B)$ is isometric to $V \oplus V^{*}$ endowed with the canonical pseudo-Euclidean form
\[
\langle X + f, Y + g \rangle_0 := f(Y) + g(X), 
\qquad \text{for all } X,Y \in V, \; f,g \in V^{*}.
\]
Hence,
\(
(\A, \prs) \simeq (V \oplus B \oplus V^*, \; \prs_n = \prs_0 + \prs_B).
\)

Since both $I$ and $I^\perp$ are two-sided ideals of $(\A, \bullet)$, the product can be written as
\begin{equation}\label{decomposition-product}
    \begin{split}
       & X \bullet Y = X \star_V Y + L(X,Y) + \theta(X,Y),  \qquad
        X \bullet u = G(X)(u) + \varphi(X,u),   \\[0.3em]
        &  
        u\bullet X=F(X)(u)+\tau(u, X), \qquad u \bullet v = u \bullet_B v + \mu(u,v),
        \\[0.3em]
          &     X \bullet f = N(X)(f), \qquad u\bullet f=\Ll_f=0,
    \end{split}
\end{equation}
for all $X, Y\in V$, $u,v\in B$ and $f\in V^*$, where $X \star_V Y \in V$, $L(X,Y), u\bullet_B v \in B$,  
\(F(X), G(X) \in \En(B) 
\), and \(
\theta(X,Y),\; \varphi(X,u),\; \tau(u,X),\;  \mu(u,v), \; N(X)(f) \in V^*\).

Since $(\A, \bullet, \prs)$ is a left-invariant pseudo-Euclidean nearly associative algebra,
  $(B, \bullet_B)$ is a nearly associative algebra, and the restriction
\(
\prs_B
\)
is a left-invariant pseudo-Euclidean scalar product on $(B, \bullet_B)$.  
Therefore, $\varphi$ satisfies Condition~\eqref{ex-central}.  
The canonical projections $\pi_\mathcal{B} : I^\perp \to \mathcal{B}$ identifies $(B, \bullet_B, \prs_B)$ with $(\mathcal{B}, \bullet_\mathcal{B}, \prs_\mathcal{B})$.  

Moreover, since $(\A, \bullet)$ is nearly associative algebra and $[\A, \A] \subseteq I^\perp$, we deduce that $[X,Y] \in I^\perp$, for all $X, Y\in V$, and therefore
\[
X \star_V Y = Y \star_V X, \qquad \text{for all } X,Y \in V.
\]
Thus $(V, \star_V)$ is a commutative associative algebra.  
Furthermore, for all $X,Y,Z \in \A$ and $f \in V^*$, we have
\[
\langle (X \bullet Y) \bullet Z, f \rangle_n
= \langle Z, (X \bullet Y) \bullet f \rangle_n
\overset{\eqref{nearly-associative}}{=} \langle Z, Y \bullet (f \bullet X) \rangle_n = 0,
\]
Since $\Ll_f = 0$, it follows that $(X \bullet Y) \bullet Z \in B \oplus V^*$, and consequently,
\[
(X \star_V Y) \star_V Z = 0.
\]
Hence, since $(V, \star_V)$ is associative, it follows that it is nilpotent of index at most $3$.  
Moreover, the canonical projection $\pi_{\h} : \A \to \h$ identifies $(V, \star_V)$ with $(\h, \star)$.

Since $\prs_n$ is a left-invariant scalar product on $(\A, \bullet)$, several useful identities follow.

First, for all $X, Y \in V$ and $f \in V^*$, we have
\[
f(X \star_V Y)
= \langle X \star_V Y, f \rangle_0= \langle X \bullet Y, f \rangle_n 
= \langle Y, X \bullet f \rangle_n
= N(X)(f)(Y),
\]
which implies that $f \circ \Ll^{\star_V}_X = N(X)(f)$.

Next, for all $X, Y \in V$ and $u \in \mathcal{B}$,
\[
\langle L(X,Y), u \rangle_{\mathcal{B}}
= \langle X \bullet Y, u \rangle_n
= \langle Y, X \bullet u \rangle_n
= \varphi(X,u)(Y).
\]

For any $X, Y, Z \in V$, one obtains
\[
\theta(X,Y)(Z)= \langle\theta(X,Y), Z \rangle_0
= \langle X \bullet Y, Z \rangle_n
= \langle Y, X \bullet Z \rangle_n
= \theta(X,Z)(Y),
\]
showing that $\theta(X,Y)$ is symmetric in the last two arguments.

Similarly, for all $X \in V$ and $u, v \in \mathcal{B}$,
\[
\langle G(X)(u), v \rangle_{\mathcal{B}}
= \langle X \bullet u, v \rangle_n
= \langle u, X \bullet v \rangle_n
= \langle u, G(X)(v) \rangle_{\mathcal{B}},
\]
which proves that $G(X)$ is symmetric with respect to $\prs_{\mathcal{B}}$.

Furthermore, for all $u, v \in \mathcal{B}$ and $X \in V$,
\[
\mu(u,v)(X)
= \langle u \bullet v, X \rangle_n
= \langle v, u \bullet X \rangle_n
= \langle v, G(X)(u) \rangle_{\mathcal{B}}.
\]

Finally, for all $X, Y \in V$ and $u \in \mathcal{B}$,
\[
\tau(u,X)(Y)
= \langle u \bullet X, Y \rangle_n
= \langle X, u \bullet Y \rangle_n
= \tau(u,Y)(X),
\]
and thus the bilinear map $\omega : V \times V \to \A$ is defined by
\[
\langle \omega(X,Y), u \rangle = \tau(u,X)(Y), 
\quad \text{for all } X,Y \in V, \; u \in \mathcal{B},
\]
is symmetric.

The product \eqref{decomposition-product} can be rewritten in the equivalent form
\begin{equation}\label{decomposition-product1}
    \begin{split}
        &X \bullet Y = X \star Y + L(X,Y) + \theta(X,Y),  \quad 
        X \bullet u = G(X)(u) + \langle L(X,\cdot), u \rangle_{\mathcal{B}},  \\[0.3em]
        & u \bullet X = F(X)(u) + \langle \omega(X,\cdot), u \rangle_{\mathcal{B}}, \quad u \bullet v = u \bullet_B v + \langle F(\cdot)(u), v \rangle_{\mathcal{B}},  \\[0.3em]
        & X \bullet f = f \circ \Ll_X^{\star},  \qquad   \Ll_f =u\bullet f=0,
    \end{split}
\end{equation}
for all $X, Y\in \h$, $u,v\in \mathcal{B}$  and $f\in \h^*$, where $G(X)$ is symmetric with respect to $\prs_{\mathcal{B}}$ and $\theta(X,Y)(Z) = \theta(X,Z)(Y)$ for all $X,Y,Z \in \h$.

By direct computation, one verifies that $(\h,\mathcal{B},F,G,L)$ is a context of an almost
semi-direct product of nearly associative algebras,  and that  $(F,G,L,\omega,\theta)$ satisfy the relations~\eqref{eq:prod-tildeA}. Consequently, $(\A, \bullet, \prs)$ is a double extension of the left-invariant pseudo-Euclidean nearly associative algebra
\(
(\mathcal{B} = I^\perp / I, \bullet_\mathcal{B}, \prs_\mathcal{B})
\)
by the  commutative associative algebra
\(
\h= \A / I^\perp,
\)
which is nilpotent of index at most $3$, by means of $(F,G,L,\omega,\theta)$.
\end{proof}

Now, we characterize left-invariant pseudo-Euclidean nearly associative algebras by using the notion of double extension by a commutative associative algebra which is
nilpotent of index at most $3$.
\begin{theo}\label{Double-generale}
    Let $(\A, \bullet, \prs)$ be a left-invariant pseudo-Euclidean non-commutative nearly associative algebra. 
    We denote 
    \(
    I := [\A, \A]\cap [\A, \A]^\perp.
    \) 
    Then $(\A, \bullet, \prs)$ is a double extension of the left-invariant pseudo-Euclidean nearly associative algebra 
    \(
    (\mathcal{B} := I^\perp / I, \bullet_\mathcal{B}, \prs_\mathcal{B})
    \)
    by the commutative associative algebra 
    \(
    \h := \A / I^\bot,
    \)
    which is nilpotent of index at most $3$, by means of  $(F, G, L, \omega, \theta)$.
\end{theo}

\begin{proof}
    Since $(\A, \bullet, \prs)$ is a left-invariant pseudo-Euclidean non-commutative nearly associative algebra, Proposition~\ref{Non-commutative} ensures that $I \neq \{0\}$. 
    By Proposition~\ref{tow-sided}, $I$ is a totally isotropic two-sided ideal of $(\A, \bullet)$. 
    Hence, by Theorem~\ref{Lemme-ex}, $(\A, \bullet, \prs)$ is a double extension of the left-invariant pseudo-Euclidean nearly associative algebra 
    \(
    (\mathcal{B} := I^\perp / I, \bullet_\mathcal{B}, \prs_\mathcal{B})
    \)
    by the commutative associative algebra 
    \(
    \h := \A / I^\perp,
    \)
    which is nilpotent of index at most $3$, by means of the data $(F, G, L, \omega, \theta)$.
\end{proof}

\begin{co}\label{suite0}
 Let $(\A, \bullet, \prs)$ be a left-invariant pseudo-Euclidean non-commutative nearly associative algebra. 
 Then $(\A, \bullet, \prs)$ is obtained by a sequence of double extensions  by the nilpotent commutative associative
algebra $(\mathfrak{h}, \star)$ of nilpotency index at most $3$ starting from a quadratic commutative associative algebra.   
\end{co}

\begin{proof}
 According to Theorem~\ref{Double-generale}, $(\A,  \bullet, \prs)$ is a double extension
of some left-invariant pseudo-Euclidean nearly associative algebra $(\mathcal{B}_1, \bullet_{\mathcal{B}_1}, \prs_{\mathcal{B}_1})$.  If $(\mathcal{B}_1, \bullet_{\mathcal{B}_1})$ is commutative, the proof is complete.
Otherwise, if $\mathcal{B}_1$ is non-commutative, then it is itself a double extension
of another left-invariant pseudo-Euclidean nearly associative algebra $(\mathcal{B}_2, \bullet_{\mathcal{B}_2}, \prs_{\mathcal{B}_2})$.  
By iterating this process, we obtain a descending chain of double extensions. Since each step strictly reduces the dimension of the algebra, this process must terminate after a finite number of steps $k \in \mathbb{N}^*$. Consequently, $(\mathcal{B}_k, \bullet_{\mathcal{B}_k})$ is a commutative algebra, and hence $(\mathcal{B}_k, \bullet_{\mathcal{B}_k}, \prs_{\mathcal{B}_k})$ is a quadratic commutative associative algebra, which concludes the proof.
\end{proof}

Let $(\A, \bullet)$ be a nearly associative algebra, and consider its underlying solvable Lie algebra $\A^- = (\A, \br)$, where $[u,v] = u \bullet v - v \bullet u$ for all $u,v\in \A$. Then, the linear map $\ad : \A \longrightarrow \En(\A)$ defined, for all $u,v\in \A$, by
\[
\ad_u(v) = [u,v],
\]
 is a representation of the Lie algebra $\A^-$ on itself. 

\begin{Le}\label{Le3}
Let $(\mathcal{A}, \bullet, \prs)$ be a left-invariant pseudo-Euclidean non-commutative nearly associative algebra over a field $\mathbb{K}$.
Then:

\begin{enumerate}
    \item[\rm $(i)$] If $\mathbb{K}$ is algebraically closed, then there exists a one-dimensional totally isotropic two-sided ideal $J$ of $(\mathcal{A}, \bullet)$ such that $J \subseteq [\mathcal{A}, \mathcal{A}] \cap [\mathcal{A}, \mathcal{A}]^\perp$. Moreover, its orthogonal complement $J^\perp$ is also a two-sided ideal of $(\mathcal{A}, \bullet)$.
    \item[\rm $(ii)$] If $\mathbb{K}=\mathbb{R}$, then there exists a totally isotropic two-sided ideal $J$ of $(\mathcal{A}, \bullet)$ of dimension $1$ or $2$ such that $J \subseteq [\mathcal{A}, \mathcal{A}] \cap [\mathcal{A}, \mathcal{A}]^\perp$. Moreover, its orthogonal complement $J^\perp$ is also a two-sided ideal of $(\mathcal{A}, \bullet)$.
\end{enumerate}
\end{Le}

\begin{proof}
Since $(\A, \bullet, \prs)$ is a left-invariant pseudo-Euclidean non-commutative nearly associative algebra, 
Proposition~\ref{Non-commutative} implies that 
\(
[\A, \A] \cap [\A, \A]^\perp \neq \{0\}.
\)
By Proposition~\ref{tow-sided}, the space $[\A, \A] \cap [\A, \A]^\perp$ is a two-sided ideal of $(\A, \bullet)$.  
Therefore,
\(
\ad : \A \longrightarrow \En([\A, \A] \cap [\A, \A]^\perp)
\)
is a representation of the solvable Lie algebra $\A^-$ on $[\A, \A] \cap [\A, \A]^\perp$.

By Lie's Theorem, if $\mathbb{K}$ is algebraically closed \textup{(}resp.~if $\mathbb{K}=\mathbb{R}$\textup{)}, there exists a minimal non-trivial Lie ideal $J \subseteq [\mathcal{A}, \mathcal{A}] \cap [\mathcal{A}, \mathcal{A}]^\perp$ of dimension $1$ \textup{(}resp.~of dimension $1$ or $2$\textup{)}.

Since $J \subseteq [\A, \A] \cap [\A, \A]^\perp$, Lemma~\ref{Le2} ensures that 
\(
\Ll_u = 0,\) for all $ u \in J,
$
and thus $J$ is a right ideal of $(\A, \bullet)$. Moreover, for all $u \in J$ and $v \in \A$, we have 
\(
[u,v] = - v \bullet u \in J,
\)
which shows that $J$ is also a two-sided ideal of $(\A, \bullet)$.  

Finally, since $J \subseteq [\A, \A] \cap [\A, \A]^\perp$ and $J$ is a two-sided ideal of $(\A, \bullet)$, 
Proposition~\ref{reduit1} implies that its orthogonal $J^\perp$ is also a two-sided ideal of $(\A, \bullet)$.
\end{proof}

\begin{theo}\label{Closed}
Let $(\mathcal{A}, \bullet, \prs)$  be a left-invariant pseudo-Euclidean non-commutative nearly associative algebra over an algebraically closed field $\mathbb{K}$. Then there exists a one-dimensional totally isotropic two-sided ideal $I$ of $(\mathcal{A}, \bullet, \prs)$ such that $I \subseteq [\mathcal{A}, \mathcal{A}] \cap [\mathcal{A}, \mathcal{A}]^\perp$. Moreover, $(\mathcal{A}, \bullet, \prs)$ is a double extension of $(\mathcal{B} := I^\perp / I, \bullet_\mathcal{B}, \omega_\mathcal{B})$ by the one-dimensional trivial algebra 
        $\h:= \A/I^\bot \simeq \K d$, by means of the data $(\delta, D, a_0, \lambda)$.
\end{theo}

\begin{proof}
 Since $(\A, \bullet, \prs)$ is a left-invariant pseudo-Euclidean non-commutative nearly associative algebra, Lemma~\ref{Le3} ensures the existence of a totally isotropic two-sided ideal $J \subseteq [\A, \A]\cap [\A, \A]^\perp$ of dimension $1$. According to Theorem~\ref{Lemme-ex}, the algebra $(\mathcal{A}, \bullet, \prs)$ is a double extension of  $(\mathcal{B} := I^\perp / I, \bullet_\mathcal{B}, \omega_\mathcal{B})$ by the one-dimensional trivial algebra 
        $\h:= \A/I^\bot \simeq \K d$, by means of the data $(\delta, D, a_0, \lambda)$.
\end{proof}

\begin{co}
 Let $(\mathcal{A}, \bullet, \prs)$  be a left-invariant pseudo-Euclidean non-commutative nearly associative algebra over an algebraically closed field $\mathbb{K}$.
 Then $(\A, \bullet, \prs)$ is obtained by a sequence of double extensions by a one-dimensional  trivial algebra starting from a quadratic commutative associative algebra.   
\end{co}

\begin{proof}
The proof follows analogously to the proof of Corollary \ref{suite0}.

\end{proof}

Now, we characterize left-invariant pseudo-Euclidean nearly associative algebras over the field $\mathbb{R}$ by using the notion of double extension by a one- or two-dimensional nilpotent commutative associative algebra.
\begin{theo}\label{db-1-2}
    Let $(\A, \bullet, \prs)$ be a left-invariant pseudo-Euclidean non-commutative nearly associative algebra over the field $\mathbb{R}$. 
    Then $(\A, \bullet, \prs)$ is either  
    \begin{itemize}
        \item[(i)] a double extension of a left-invariant pseudo-Euclidean nearly associative algebra $(\mathcal{B}, \bullet_\mathcal{B}, \prs_\mathcal{B})$ by the one-dimensional trivial algebra 
        $\h \simeq \mathbb{R} d$, by means  $(\delta, D, a_0, \lambda)$; or  
        \item[(ii)] a double extension of a left-invariant pseudo-Euclidean nearly associative algebra $(\mathcal{B}, \bullet_\mathcal{B}, \prs_\mathcal{B})$ by the two-dimensional nilpotent commutative associative algebra 
        $\h\simeq \mathbb{R} d_1 \oplus \mathbb{R} d_2$, by means of $(\delta_i, D_i, a_{ij}, b_{ij}, \al_{ij}, \beta_{ij})$, for all $i,j \in \{1,2\}$.  
    \end{itemize}
\end{theo}

\begin{proof}
    Since $(\A, \bullet, \prs)$ is a left-invariant pseudo-Euclidean non-commutative nearly associative algebra, Lemma~\ref{Le3} ensures the existence of a totally isotropic two-sided ideal $J \subseteq [\A, \A]\cap [\A, \A]^\perp$ of dimension $1$ or $2$.  According to Theorem~\ref{Lemme-ex}, we distinguish two cases:

    If $\dim J = 1$, then $(\A, \bullet, \prs)$ is a double extension of the left-invariant pseudo-Euclidean nearly associative algebra $(\mathcal{B}:=J^\perp/J, \bullet_\mathcal{B}, \prs_\mathcal{B})$ by the one-dimensional trivial algebra $\h:=\A/J^\perp\simeq \mathbb{R} d$, by means of  $(\delta, D, a_0, \lambda)$.  

    If $\dim J = 2$, then $(\A, \bullet, \prs)$ is a double extension of the left-invariant pseudo-Euclidean nearly associative algebra $(\mathcal{B}:=J^\perp/J, \bullet_\mathcal{B}, \prs_\mathcal{B})$ by the two-dimensional nilpotent commutative associative algebra $\h:=\A/J^\perp\simeq \mathbb{R} d_1 \oplus \mathbb{R} d_2$ of nilpotent index at most 3, by means of  $(\delta_i, D_i, a_{ij}, b_{ij}, \al_{ij}, \beta_{ij})$, for all $i,j \in \{1,2\}$.
\end{proof}
\begin{co}\label{Lor1}
Let $(\A, \bullet, \prs)$ be a left-invariant Lorentzian non-commutative nearly associative algebra. Then $(\A, \bullet, \prs)$ is a double extension of a left-invariant Euclidean nearly associative algebra  $(\mathcal{B}, \bullet_\mathcal{B}, \prs_\mathcal{B})$ by a one-dimensional trivial algebra 
$\h \simeq \mathbb{R} d$, by means of $(\delta, D, a_0, \lambda)$.  
\end{co}

\begin{proof}
 Since $(\A, \bullet, \prs)$ is a left-invariant pseudo-Euclidean non-commutative nearly associative algebra, Lemma~\ref{Le3} ensures the existence of a totally isotropic two-sided ideal $J \subseteq [\A, \A]\cap [\A, \A]^\perp$ of dimension $1$ or $2$.  
Now, by Proposition~\ref{Non-commutative}, we have $[\A, \A]\cap [\A, \A]^\perp \neq \{0\}$.  
Since $\prs$ is Lorentzian, it follows that 
$\dim \big( [\A, \A]\cap [\A, \A]^\perp \big) = 1$ , and as a consequence $J=[\A, \A]\cap [\A, \A]^\perp$. Hence, by Theorem~\ref{Lemme-ex}, 
$(\A, \bullet, \prs)$ is a double extension of the left-invariant pseudo-Euclidean nearly associative algebra $(\mathcal{B}:=J^\perp/J, \bullet_\mathcal{B}, \prs_\mathcal{B})$ by the one-dimensional trivial algebra $\h:= \A/J^\bot \simeq \mathbb{R} d$, by means of $(\delta, D, a_0, \lambda)$. Moreover, since $(\A, \prs)$ is Lorentzian, it follows that $(\mathcal{B}, \prs_{\mathcal{B}})$ is Euclidean. 
\end{proof}
\begin{co}
 Let $(\A, \bullet, \prs)$ be a left-invariant pseudo-Euclidean non-commutative nearly associative algebra over the field $\mathbb{R}$. 
 Then $(\A, \bullet, \prs)$ is obtained by a sequence of double extensions by a one-dimensional or a two-dimensional starting from a quadratic commutative associative algebra.   
\end{co}

\begin{proof}
The proof follows analogously to the proof of Corollary \ref{suite0}.
\end{proof}

\section{Classification of left-invariant  Lorentzian non-commutative nearly associative algebras of dimension $\leq 4$}\label{s:class left-inv}

In \cite{BBEL}, the authors provided the classification of Euclidean left-invariant left-alternative algebras. It is well known that any commutative left-alternative algebra is a commutative associative algebra. Consequently, by virtue of \cite[Theorem 3.5]{BBEL}, we deduce a complete classification of Euclidean (or positive-definite quadratic) left-invariant commutative associative algebras in the following proposition.

\begin{pr}[\cite{BBEL}]\label{cla-Q}
Let $(\A, \bullet, \prs)$ be a quadratic Euclidean (or positive-definite quadratic) left-invariant commutative associative algebra. Then there exists an orthogonal basis $\{e_1, \dots, e_n, f_1, \dots, f_p\}$ of $\A$ such that the only non-vanishing products among the basis elements are given by
$$e_i \bullet e_i = e_i, \quad \text{for } i = 1, \dots, n,$$
where the scalar product $\prs$ satisfies:
$$\langle e_i, e_i \rangle = \lambda_i > 0, \qquad \langle f_j, f_j \rangle = 1,$$
for all $i = 1, \dots, n$ and $j = 1, \dots, p$.
\end{pr}

By Corollary~\ref{Lor1}, every $4$-dimensional left-invariant Lorentzian nearly associative algebra arises from a double extension of a $2$-dimensional left-invariant Euclidean commutative associative algebra $(\mathcal{B}, \bullet_{\mathcal{B}}, \prs_{\mathcal{B}})$ by a one-dimensional trivial algebra $\h$ by means of $(\de, D, a_0, b_0,  \la)$.

Accordingly, let $(\mathcal{B},\bullet_{\mathcal{B}},\prs_{\mathcal{B}})$ be a $2$-dimensional left-invariant Euclidean commutative associative algebra. Our goal is to determine all linear maps $\de, D:\mathcal{B} \to \mathcal{B}$, where $D$ is symmetric with respect to $\prs_{\mathcal{B}}$, and elements $a_0, b_0 \in \mathcal{B}$ such that $(\de, D, a_0,b_0, \lambda)$ satisfies system~\eqref{sys-dim1}.

\begin{pr}
\begin{enumerate}
    \item 
There exists no $2$-dimensional   left-invariant  Lorentzian non-commutative nearly associative algebra. 

\item Every $3$-dimensional  left-invariant  Lorentzian non-commutative nearly associative algebra is isomorphic to $(\A^3,\bullet,\prs)$, where  the only non-trivial products among the elements of a basis $\{d,e_1,e\}$ are given by 
$$
d\bullet d=\lambda e,\qquad
e_1\bullet d=e,
$$
with $\la \in \mathbb{R}$, and  the left-invariant scalar product is defined by 
$$
\prs=e^*\odot d^*+e_1^*\otimes e_1^*.
$$  
\end{enumerate}
\end{pr}

\begin{proof}
Assume first that $(\A,\bullet,\prs)$ is a  left-invariant  $2$-dimensional Lorentzian non-commutative nearly associative algebra. By Corollary~\ref{Lor1}, $(\A,\bullet,\prs)$ is a double extension of the trivial algebra $\{0\}$ by means of $(0,0,0,0,\lambda)$. It follows from the product given by \eqref{Produit-dim1} that $(\A,\bullet)$ is commutative, which contradicts the assumption. Hence, no such algebra exists.

Suppose now that $(\A,\bullet,\prs)$ is a left-invariant $3$-dimensional Lorentzian non-commutative nearly associative algebra. By Corollary~\ref{Lor1}, $(\A,\bullet,\prs)$ is a double extension of a one-dimensional Euclidean nearly associative algebra $(\mathcal B,\bullet_\mathcal{B}, \prs_{\mathcal{B}})$ by means of $(\delta,D,a_0,b_0,\lambda)$.

Suppose first that $(\mathcal B,\bullet_\mathcal{B})$ is a trivial algebra. Then System~\eqref{sys-dim1} yields
$
D^2=\delta^2=0.
$ Since $\dim\mathcal B=1$, we obtain $D=\delta=0$. Consequently, System~\eqref{sys-dim1} reduces to
$
\langle a_0,a_0-b_0\rangle_{\mathcal B}=0.$ Let $\mathcal B=\mathbb Re_1$ with
$
\langle e_1,e_1\rangle_{\mathcal B}=1.
$
Write
$
a_0=\alpha e_1,$ and $b_0=\beta e_1,$ where $\alpha,\beta\in\mathbb R$. Then
$
\alpha^2=\alpha\beta.
$
Hence, either $\alpha=0$ or $\alpha=\beta\neq 0$. 

If $\alpha=\beta$, then the product $\bullet$ of $\A$, which is given by \eqref{Produit-dim1}, is commutative, contradicting the hypothesis. Therefore, $\alpha=0$. Using the product given by \eqref{Produit-dim1}, there exists a basis $\{d,e_1,e\}$ of $\A$ such that
$$
d\bullet d=\lambda e,
\quad
e_1\bullet d=\beta e,
\esp 
\prs=e^*\odot d^*+e_1^*\otimes e_1^*.
$$
Since $(\A,\bullet)$ is non-commutative, necessarily $\beta\neq0$. Hence, $(\A,\bullet,\prs)$ is isomorphic to $(\A^3,\bullet,\prs)$.

Assume next that $(\mathcal B,\bullet_\mathcal{B})$ is non-trivial. By Proposition~\ref{cla-Q}, $\mathcal B=\mathbb Re_1$ with
$$
e_1\bullet_\mathcal{B} e_1=e_1,\esp 
\langle e_1,e_1\rangle_{\mathcal B}=\mu>0.
$$
System~\eqref{sys-dim1} implies that
$
\delta(e_1)=D(e_1)=xe_1,$ 
and 
$a_0=b_0=x^2e_1,$ for some $x\in\mathbb R$.
Since the product $\bullet$ of $\A$ is given by \eqref{Produit-dim1}, it follows that $(\A,\bullet)$ is commutative, contradicting the assumption that $(\A,\bullet)$ is non-commutative. Therefore, this case cannot occur.
\end{proof}

\begin{Le}\label{Le-s1}
Let $(\mathcal{B}, \prs_{\mathcal{B}})$ be a $2$-dimensional Euclidean vector space. Then, $(D, \delta, a_0, b_0)$ satisfies system~\eqref{sys-dim1} if and only if $D = \delta = 0$ and there exists an orthonormal basis $\mathcal{B}_0 = \{e_1, e_2\}$ of $\mathcal{B}$ such that either one of the following holds:
\begin{enumerate}
    \item[\rm (a)] $a_0 = 0$,  $b_0 = \alpha e_1$, and $\la\in \mathbb{R}$, with $\alpha \in \mathbb{R}$.
    \item[\rm (b)] $a_0 = \alpha e_1$,  $b_0 = \alpha e_1 + \beta e_2$, and $\la\in \mathbb{R}$, with  $\alpha \in \mathbb{R}^*$ and $\beta \in \mathbb{R}$.
\end{enumerate}
\end{Le}

\begin{proof}
The "only if" implication is immediate.

We now prove the converse. 
System~\eqref{sys-dim1} yields $D^2 = 0$ and $\delta^* \circ \delta = 0$. Hence, as $\langle \cdot, \cdot \rangle_{\mathcal{B}}$ is Euclidean, $\delta^* \circ \delta = 0$ implies $\delta = 0$, and the symmetry of $D$ yields $D = 0$. Thus, system~\eqref{sys-dim1} reduces to $\langle a_0, a_0 - b_0 \rangle_{\mathcal{B}} = 0$. 

If $a_0 = 0$, the relation holds trivially, by choosing an orthonormal basis $\mathcal{B}_0 = \{e_1, e_2\}$ such that  $b_0 = \alpha e_1$ where $\alpha \in \mathbb{R}$, proving (a). 

If $a_0 \neq 0$, we choose an orthonormal basis $\mathcal{B}_0 = \{e_1, e_2\}$ such that $a_0 = \alpha e_1$ where $\alpha \in \mathbb{R}^*$. Setting $b_0 = \beta e_1 + \gamma e_2$, the relation becomes $\alpha(\alpha - \beta) = 0$. Since $\alpha \neq 0$, we get $\beta = \alpha$, which establishes (b).
\end{proof}

According to Proposition~\ref{cla-Q}, every non-trivial $2$-dimensional associative  left-invariant Euclidean algebra is isomorphic to one of the following algebras:
\begin{enumerate}
\item $(\mathcal B_1^2, \bullet,\prs)$,  where $\mathcal B_1^2$ admits a basis $\{e_1,e_2\}$ satisfying   
$$
e_1\bullet e_1=e_1, \;   \Ll_{e_2}=\Rr_{e_2}=0,  \esp \prs=\mu e_1^*\otimes e_1^*
+e_2^*\otimes e_2^*,
\quad \mu>0.
$$
\item $(\mathcal B_2^2,\bullet, \prs)$,  where $\mathcal B_2^2$ admits a basis $\{e_1,e_2\}$ satisfying    
$$
e_1\bullet e_1=e_1,
\;
e_2\bullet e_2=e_2, \; e_1 \bullet e_2=e_2 \bullet e_1=0,
\esp
\prs=\mu_1 e_1^*\otimes e_1^*
+\mu_2 e_2^*\otimes e_2^*,
\quad
\mu_1,\mu_2>0.
$$
\end{enumerate}

\begin{Le}\label{Le-s2}
 \begin{enumerate}
     \item $( \de, D, a_0, b_0, \la)$ on $(\mathcal B^2_1,\prs)$ satisfies \eqref{sys-dim1} if and only if \begin{enumerate}
\item[$(a_1)$] $\mathrm{Mat}(\delta,\mathbb{B})=\mathrm{Mat}(D,\mathbb{B})=
\begin{pmatrix}
x & 0\\
0 & 0
\end{pmatrix}$, \, $a_0=b_0=x^2 e_1$, and  $\la\in \mathbb{R}$, where $x\in\mathbb{R}$.
\item[$(a_2)$] $\mathrm{Mat}(\delta,\mathbb{B})=\mathrm{Mat}(D,\mathbb{B})=
\begin{pmatrix}
x & 0\\
0 & 0
\end{pmatrix}$, \, $a_0=b_0=x^2 e_1+\al e_2$, and $\la\in \mathbb{R}$, where $x\in\mathbb{R}$, $\al\in \mathbb{R}^*$.
\end{enumerate}
 \item $( \de, D, a_0, b_0, \la)$ on $(\mathcal B^2_2,\prs)$ satisfies \eqref{sys-dim1} if and only if 
 $$\mathrm{Mat}(\delta,\mathbb{B})=\mathrm{Mat}(D,\mathbb{B})=
\begin{pmatrix}
x & 0\\
0 & y
\end{pmatrix}, \, a_0=b_0=x^2 e_1+y^2e_2, \text{ and }  \la\in \mathbb{R}, \text{ where } x, y\in\mathbb{R}.$$
     \end{enumerate} 
\end{Le}

    \begin{proof}
The proof follows from a straightforward computation using System~\eqref{sys-dim1}.
\end{proof}

\begin{pr}

Every $4$-dimensional Lorentzian non-commutative nearly associative algebra is isomorphic to one of the following algebras:
     \begin{enumerate}
         \item $\A^4_1=(\spa(d,e_1,e_2, e),\bullet,\prs)$, where the only non-trivial products among the generators are given by 
         $$
d\bullet d=\lambda e,
\quad
e_1\bullet d= e,
$$
with  $\la \in \mathbb{R}$,   and   the left-invariant scalar product is defined by 
$$
\prs
=
e^*\odot d^*
+
e_1^*\otimes e_1^*
+
e_2^*\otimes e_2^*.
$$

\item $\A^4_2=(\mathrm{span}(d,e_1,e_2, e),\bullet,\prs)$, where  the only non-trivial products among the generators are given by  
$$
 d\bullet d=  e_1 + \lambda e,
\quad
d\bullet e_1=e_1\bullet d=
e_2\bullet d=  e,  
$$
with  $\la \in \mathbb{R}$, and  the left-invariant scalar product is defined by 
$$
\prs
=
e^*\odot d^*
+
e_1^*\otimes e_1^*
+
e_2^*\otimes e_2^*.
$$
\end{enumerate}
\end{pr}

\begin{proof}
According to Corollary~\ref{Lor1}, $(\A,\bullet,\prs)$ is a double extension of a 2-dimensional Euclidean nearly associative algebra $(\mathcal B,\bullet_{\mathcal B}, \prs_{\mathcal B})$ by means of $(\delta,D,a_0,b_0,\lambda)$.

Suppose first that $(\mathcal B,\bullet_{\mathcal B})$ is trivial. According to Lemma~\ref{Le-s1}, the system~\eqref{sys-dim1} yields one of the following possibilities.

In Case (a), using the product $\bullet$ given by \eqref{Produit-dim1} and  using the bilinear form defined by \eqref{metric-dim1}, there exists a basis $\{d,e_1,e_2,e\}$ of $\A$ such that
$$
d\bullet d=\lambda e,
\quad
e_1\bullet d=\alpha e,  \quad e_2 \bullet d=d \bullet e_i=d\bullet e_i= e_i \bullet e_j=\Ll_e=\Rr_e=0,
$$
for $i,j\in \{1,2\}$, with $\al,\la\in \mathbb{R}$, and
$$
\prs
=
e^*\odot d^*
+
e_1^*\otimes e_1^*
+
e_2^*\otimes e_2^*.
$$
Since $(\A,\bullet)$ is non-commutative, necessarily $\alpha\neq0$. Hence, $(\A,\bullet,\prs)$ is isomorphic to $(\A_1^4,\bullet,\langle\cdot,\cdot\rangle)$.

In Case (b), using the product given by \eqref{Produit-dim1} and  using the bilinear form defined by \eqref{metric-dim1}, there exists a basis $\{d,e_1,e_2,e\}$ of $\A$ such that
$$
d\bullet d=\alpha e_1 + \lambda e,
\quad
d\bullet e_1=e_1\bullet d=  \alpha  e,
\quad
e_2\bullet d=\beta e,  \quad d \bullet e_2= e_i \bullet e_j=\Ll_e=\Rr_e=0, 
$$
for $i,j\in \{1,2\}$ with $\la,\be\in \mathbb{R}$, $\al\in \mathbb{R}^*$, and
$$
\prs
=
e^*\odot d^*
+
e_1^*\otimes e_1^*
+
e_2^*\otimes e_2^*.
$$
Since $(\A,\bullet)$ is non-commutative, necessarily $\beta\neq0$. Hence, $(\A,\bullet,\langle\cdot,\cdot\rangle)$ is isomorphic to $(\A_2^4,\bullet,\prs)$.

Now assume that $(\mathcal B,\bullet_{\mathcal B})$ is non-trivial. According to Lemma~\ref{Le-s2}, we obtain in all cases that $D=\delta$ and $a_0=b_0$. Using the product given by \eqref{Produit-dim1}, it follows that $(\A,\bullet)$ is commutative, which is a contradiction. Hence this case cannot occur.
\end{proof}

\end{document}